\documentclass[11pt, reqno]{amsart}

\usepackage[dvipsnames,usenames]{color}
\usepackage[colorlinks=true, urlcolor=NavyBlue, linkcolor=NavyBlue, citecolor=NavyBlue]{hyperref}
\usepackage{float}
\usepackage{amsmath, amsthm, amssymb}
\usepackage{amsfonts}
\usepackage{enumerate}
\usepackage{epsf}
\usepackage{psfrag}
\usepackage{amsmath}
\usepackage[all]{xy}
\usepackage[dvips]{graphicx}
\usepackage{epsfig}
\usepackage{epstopdf}
\usepackage{ifpdf}
\usepackage[UKenglish]{babel}
\usepackage{subfigure}
\usepackage{hyperref}
\usepackage{color}
\usepackage{marginnote}

\hypersetup{
linkbordercolor={1 0 0}, 
citebordercolor={0 1 0} 
}
\newtheorem{thm}{Theorem}[section]
\newtheorem{lemma}[thm]{Lemma}
\newtheorem{question}[thm]{Question}
\newtheorem{prop}[thm]{Proposition}
\newtheorem{cor}[thm]{Corollary}

\newtheorem{conj}[thm]{Conjecture}
\newtheorem*{theorem*}{Claim}

 \theoremstyle{definition}
\newtheorem{defn}[thm]{Definition}
\theoremstyle{remark}
\newtheorem{rmk}[thm]{Remark}

\newcommand{\rank}{\textup{rk }}

\newcommand{\spin}{\mathrm{Spin}^c}

\makeatletter

\newcommand{\Rmnum}[1]{\expandafter\@slowromancap\romannumeral #1@}
\makeatother
\author[Yi Ni]{Yi Ni}
\thanks{}
\address {Department of Mathematics, California Institute of Technology, Pasadena, CA 91125}
\email{yini@caltech.edu}

\author[Faramarz Vafaee]{Faramarz Vafaee}
\thanks{}
\address {Department of Mathematics, California Institute of Technology, Pasadena, CA 91125}
\email{vafaee@caltech.edu}

\begin{document}
\title{Null surgery on knots in L-spaces}

\begin{abstract}
Let $K$ be a knot in an L-space $Y$ with a Dehn surgery to a surface bundle over $S^1$. We prove that $K$ is rationally fibered, that is, the knot complement admits a fibration over $S^1$. As part of the proof, we show that if $K\subset Y$ has a Dehn surgery to $S^1 \times S^2$, then $K$ is rationally fibered. In the case that $K$ admits some $S^1 \times S^2$ surgery, $K$ is Floer simple, that is, the rank of $\widehat{HFK}(Y,K)$ is equal to the order of $H_1(Y)$. By combining the latter two facts, we deduce that the induced contact structure on the ambient manifold $Y$ is tight.

In a different direction, we show that if $K$ is a knot in an L-space $Y$, then any Thurston norm minimizing rational Seifert surface for $K$ extends to a Thurston norm minimizing surface in the manifold obtained by the null surgery on $K$ (i.e., the unique surgery on $K$ with $b_1>0$).
\end{abstract}

\maketitle

\section{Introduction}\label{intro}
Heegaard Floer homology, introduced by Ozsv\'{a}th and Szab\'{o}, produces a package of invariants of three- and four-dimensional manifolds~\cite{OSzAnn1}. One example is $\widehat{HF}(Y)$, that associates a graded abelian group to a closed three-manifold $Y$. When $Y$ is a rational homology sphere, $\rank \widehat{HF}(Y) \ge \left |H_1(Y)\right |$~\cite{Ozsvath2004a}. If equality is achieved, $Y$ is called an \emph{L-space}. The name stems from the fact that lens spaces are L-spaces. More generally, all connected sums of manifolds with elliptic geometry are L-spaces~\cite{Ath}.

\subsection{Knots in L-spaces with fibered surgeries}

In an unpublished manuscript~\cite{Berge} Berge gave a conjecturally complete list of knots in $S^3$ admitting lens space fillings. The Berge conjecture roots in the classification of lens space surgeries on torus knots \cite{Moser1971}, followed by notable examples of lens space fillings on non-torus knots \cite{Rolfsen1977, Berge1991, Bleiler1989, FS1980, Gabai1989, Gabai1990, Wang1989, Wu1990}. In recent years, techniques from Heegaard Floer homology were applied to give deeper insight on the fiberedness~\cite{Ni2009}, positivity~\cite{Hedden2007}, and various notions of simplicity of knots in $S^3$ with lens space, or more generally L-space, surgeries~\cite{Ath, Hedden2011,HeddenBerge,RasmussenBerge}. The theme of the present work, in part, is to study the analogous properties of such knots when $S^3$ is replaced by $S^1 \times S^2$. It is often convenient to view the problem from the perspective of surgery along a knot in an L-space. Note that a knot $L\subset S^1 \times S^2$ on which Dehn surgery yields an L-space $Y$, induces a dual knot $K\subset Y$, the core of the surgery solid torus. By removing the interior of a neighborhood of $K\subset Y$ and undoing the original Dehn surgery, it follows that $K$ admits a surgery producing $S^1 \times S^2$. One way to obtain an example of a knot in an L-space with some $S^1 \times S^2$ surgery is as follows. Start with a solid torus $V = S^1 \times D^2$ with meridian $\mu$. Let $K\subset V$ be a \emph{Berge--Gabai} knot, i.e., $K$ has a non-trivial solid torus filling~\cite{Gabai1990}. Therefore, there is a slope $\alpha$ such that $V' = V_\alpha(K)$ is another solid torus with meridian $\mu'\ne\mu$. Note that Dehn filling $V$ along $\mu'$ will give a lens space $Z$. Then $K$, when viewed as a knot in $Z$, has an $S^1 \times S^2$ surgery; namely, $Z_\alpha(K)$ has a genus one Heegaard splitting with the property that the meridians of the two solid tori coincide (this common meridian is $\mu'$).

In~\cite{Baker2013}, Baker, Buck, and Lecuona proposed a classification of knots in $S^1 \times S^2$ with a longitudinal surgery to a lens space. Cebanu proved that the complement of a knot in $S^1 \times S^2$ that has a lens space filling, admits a fibration over the circle~\cite[Theorem~3.7.1]{Cebanu2012}. More precisely, he first proved that any knot $K$ in a lens space $Y$ with some $S^1 \times S^2$ surgery is \emph{Floer simple}. Moreover, $K$, as a knot in the lens space $Y$, lies in the homology class of a simple knot with some $S^1 \times S^2$ surgery. (See~\cite{HeddenBerge} for the definition of a simple knot in a lens space.) Such a simple knot is \emph{a priori} known to be fibered. Finally, he appealed to the fact that the complement of a Floer simple knot $K$ in the lens space $Y$ admits a fibration over $S^1$ if and only if the simple knot in the homology class of $[K]$ has a fibered complement over $S^1$~\cite[Corollary~5.3]{Ni2014}. We point out that Cebanu proved his result by checking all simple knots in lens spaces admitting $S^1\times S^2$ surgeries are fibered, and therefore, his proof is specific to the case of a lens space (and not an L-space in general). Building up on the work of J.~and~S.~Rasmussen~\cite{Rasmussen2015}, we give a novel proof of the more general case (obtained by replacing lens spaces with L-spaces).
\begin{thm}\label{thm:S1S2fiber}
Suppose $L\subset S^1\times S^2$ is a knot with some L-space surgery. Then the complement of $L$ in $S^1 \times S^2$ admits a fibration over $S^1$.
\end{thm}
\noindent If we replace $S^1\times S^2$ with $S^3$ in~Theorem~\ref{thm:S1S2fiber}, then we get the well-known result that a knot in $S^3$ which admits an L-space surgery is fibered~\cite[Corollary~1.3]{Ni2007}.

A knot $K\subset Y$ is {\it Floer simple} if $\rank\widehat{HFK}(Y,K)=\rank\widehat{HF}(Y)$. Floer simple knots in L-spaces often appear in the problem of L-space surgery. For example, if the $p$--surgery on a knot $L\subset S^3$ yields an L-space $Y$, then the dual knot of the surgery will be a Floer simple knot in $Y$ provided that $p$ is an integer greater than $2g(L)-1$ \cite{HeddenBerge,RasmussenBerge}. It turns out that a similar result holds in the case of $S^1\times S^2$ in place of $S^3$:
{\prop \label{prop:FloerSimple} If $K$ is a knot in an L-space $Y$ with some $S^1 \times S^2$ surgery, then $K$ is Floer simple. }
\\
{\defn Let $K$ be a rationally null-homologous oriented knot in an oriented closed three-manifold $Y$, $\nu(K)$ be a tubular neighborhood of $K$, and $\nu^{\circ}(K)$ denote the interior of $\nu(K)$. A properly embedded oriented surface $F\subset Y\setminus \nu^{\circ}(K)$ is called a {\it rational Seifert surface} for $K$, if $\partial F$ consists of coherently oriented parallel curves on $\partial\nu(K)$, $F$ has no closed component, and the orientation of $\partial F$ is coherent with the orientation of $K$. The knot $K$ is {\it rationally fibered} if the complement of $K$ in $Y$ fibers over $S^1$. In this paper, we often omit ``rationally'' when a knot is rationally fibered.}
\\

It is a well-known fact that, up to isotopy, there exists a unique simple closed curve $\alpha$ on $\partial \nu(K)$ with the property that the surgery on $K$ with slope $\alpha$ produces a manifold with the first Betti number one higher than that of $Y$. For simplicity of referring to this slope in the paper, we make the following definition:

{\defn\label{defn:NullSlope}For a knot $K\subset Y$, let $\alpha$ be the unique slope on $\partial \nu(K)$ that is rationally null-homologous in $Y\setminus \nu^{\circ}(K)$. We call $\alpha$ the \emph{null slope} of $K$ in $Y$. We also define the \emph{null surgery on $K$} to be Dehn filling the exterior of $K$ in $Y$ along the curve $\alpha$ and denote it $Y_\alpha$.}
\\

Theorem~\ref{thm:S1S2fiber} can be stated in terms of the dual knot $K$ of $L$ inside the L-space. More precisely, for an L-space $Y$, if the null surgery on $K\subset Y$ results in $S^1 \times S^2$, then $K$ is fibered. We generalize Theorem~\ref{thm:S1S2fiber} by replacing $S^1 \times S^2$ with an oriented closed three-manifold that is a surface bundle over $S^1$. Compare the following theorem with \cite[Corollary~1.4]{Ni2007}.

{\thm\label{SurfaceBundle} Let $K$ be a knot in an L-space $Y$. If the null surgery on $K$ is a surface bundle over $S^1$, then the complement of $K$ in $Y$ admits a fibration over $S^1$.}
\\

The above theorem does not hold for an arbitrary rational homology sphere $Y$. For example, we can choose a knot $K'\subset S^1\times S^2$ with nonzero winding number, such that the complement of $K'$ is not a surface bundle over $S^1$. Then any nontrivial surgery on $K'$ will be a rational homology sphere $Y$, and the null surgery on the dual knot $K\subset Y$ is $S^1\times S^2$, while $K$ is not rationally fibered.

When $K$ is a null-homologous knot in $Y$, Theorem~\ref{NormMinimizing} is just \cite[Corollary~1.4]{Ni2007}.
The idea of the proof of Theorem~\ref{SurfaceBundle} is inspired from that of \cite[Corollary~4.5]{Ozsvath2004}; also, a similar idea is used to prove~\cite[Corollary~1.4]{Ni2007}. The heart of the argument lies in showing that, for an appropriately chosen Spin$^c$ structure, the plus version of Heegaard Floer homology of $Y_\alpha$ is isomorphic to the hat version of knot Floer homology of $K$ in its bottommost Alexander grading. This is achieved by comparing two exact triangles which differ at only one vertex, and the groups at these distinguished vertices are the two homology groups we aim to prove are isomorphic.
See Section~\ref{section:Background} for the relevant definitions.

Since we work with rationally null-homologous knots instead of null-homologous knots as in \cite[Corollary~1.4]{Ni2007}, we encounter new difficulties. One difficulty is that the null slope is not necessarily a framing, thus we do not directly have the exact triangles we want. To solve this problem, we use a trick from \cite{Ozsvath2010} to present the null-surgery as a Morse surgery on the connected sum of $K$ and a knot in a lens space. A simple combinatorial argument (Corollary~\ref{cor:ReduceToMorse}) shows that we can reduce the general case to this special case of Morse surgery.
Another difficulty is that different Spin$^c$ structures over $Y$ may intertwine in the maps of the exact triangles. To solve this problem, we need to carefully analyze the Spin$^c$ structures. A key technical result we use is Lemma~\ref{lem:OneToOne}, which controls the interwining of the Spin$^c$ structures.

Since $Y_\alpha$ is a surface bundle over $S^1$, its Floer homology in the specified Spin$^c$ structure is of rank one. Therefore, the knot Floer homology of $K$ in its bottommost grading will be of rank one. That is, $K$ is fibered. Following from the proof of Theorem~\ref{SurfaceBundle}, we get that:
{\thm\label{NormMinimizing} Let $K$ be a knot in an L-space $Y$ with the null slope $\alpha$. For a Thurston norm minimizing rational Seifert surface $F$ of $K$, the extension $\widehat F$ in $Y_\alpha$ of $F$ is also Thurston norm minimizing.
}
\\

\noindent Theorem~\ref{NormMinimizing} generalizes a similar result of Gabai \cite[Corollary~5]{Gabai1986} where $Y = S^3$.

\subsection{Fibered, Floer simple knots and the rational-valued $\tau$ invariant}

In~\cite{Ozsvath2003}, Ozsv\'{a}th and Szab\'{o} introduced an invariant $\tau(K)$ associated to a knot $K \subset S^3$. (See also~\cite{Rasmussen2003}.) In Section~\ref{section:Background}, we define this invariant for a knot in a rational homology sphere $Y$, analogous to the integer-valued invariant in the case $Y=S^3$. The difference, in this more general setting, is that there will be as many $\tau$ invariants as the number of Spin$^c$ structures on $Y$. Moreover, since the invariant, by definition, is a function of the Alexander grading of the generators of $\widehat{CFK}(Y, K)$, the values that $\tau$ takes will be rational.

In \cite{Ni2009}, the first author defines an affine function
\[
\mathfrak H : \underline{\textup{Spin}^c}(Y, K) \to H^2(Y, K; \mathbb Q),
\]
which is basically one half of the first Chern class, shifted by an appropriate cohomology class.
The knot Floer homology provides a function
\[
y : H_2(Y, K; \mathbb Q) \to \mathbb Q,
\]
defined by
\begin{equation} \label{WidthEquation}
y(h) = \max_{\left \{ \xi \in \underline{\textup{Spin}^c}(Y, K) : \widehat{HFK}(Y, K, \xi) \ncong 0 \right \}}\langle \mathfrak H (\xi), h \rangle.
\end{equation}
When $Y=S^3$ and $h$ is a generator of $H_2(Y, K; \mathbb Z)$ (e.g. represented by a Seifert surface for $K$), it follows that $y(h)= g(K)$. If $K\subset Y$ is fibered (e.g. when $Y$ is an L-space and $K$ admits an $S^1 \times S^2$ surgery; c.f. Theorem~\ref{thm:S1S2fiber}),  we get a contact structure $\xi_K$ compatible with the rational open book decomposition specified by $(Y, K)$.

\begin{prop}
\label{prop:TightnessGeneralization} Let $K$ be a fibered, Floer simple knot in a rational homology sphere $Y$, endowed with a rational Seifert surface $F$. The following two equivalent statements hold:
\begin{itemize}
\item[(1)] The contact structure induced by the rational open book decomposition corresponding to the fibration of $(Y, K)$ is tight.
\item[(2)] For some $\textup{Spin}^c$ structure $\mathfrak{s}$ of $Y$, $\displaystyle \tau(Y,K,\mathfrak s) = \frac{y([F])}{|[\mu]\cdot[\partial F]|}$, where $\tau(Y,K,\mathfrak s)$ is as in Definition~\ref{def:RationalTau}, $[F] = h \in H^2(Y, K; \mathbb Q)$ is the homology class needed in Equation~\eqref{WidthEquation}, and $[\mu]\cdot[\partial F]$ is the intersection number of the meridian $\mu$ of $K$ with $\partial F$ in $\partial\nu(K)$.
\end{itemize}
\end{prop}

When $Y=S^3$, Proposition~\ref{prop:TightnessGeneralization} reduces to~\cite[Items~(2)~and~(4)~of~Proposition~2.1]{Hedden2007}. The main ingredient used in the proof of Proposition~\ref{prop:TightnessGeneralization} is the non-vanishing of the Heegaard Floer contact invariant associated to $K$. Hedden and Plamenevskaya, in \cite{Hedden2011}, introduced a contact invariant for a fibered knot $K$ in a closed three-manifold $Y$. The invariant is the image of the generator of the homology of the bottom filtered subcomplex in the Heegaard Floer homology of $Y$ under the natural map
\[
\widehat{HFK}(-Y, K,\mathrm{bottommost}) \to \widehat{HF}(- Y),
\]
where $- Y$ is the manifold $Y$ with opposite orientation. To prove Proposition~\ref{prop:TightnessGeneralization}, it will be straightforward to check that the Heegaard Floer contact invariant associated to $K$ is non-zero, and therefore, the contact structure induced by $K$ is tight \cite{Hedden2011}.

From the proof of Proposition~\ref{prop:TightnessGeneralization}, we get the following corollary that may be of independent interest:
\begin{cor} \label{FloerSimpleTau} Let $K$ be a fibered, Floer simple knot in a rational homology sphere $Y$, endowed with a rational Seifert surface $F$. There exists a Spin$^c$ structure $\mathfrak{s}$ on $Y$ such that $$\displaystyle \tau(Y,K,\mathfrak s)= \frac{y([F])}{|[\mu]\cdot[\partial F]|}.$$
\end{cor}

Combining Theorem~\ref{thm:S1S2fiber} and Propositions~\ref{prop:TightnessGeneralization},~\ref{prop:FloerSimple}, we get the following theorem:
\begin{thm} \label{thm:S1S2EquivState} Let $K$ be a knot in an L-space $Y$ such that $K$ admits an $S^1 \times S^2$ surgery. Let also $F$ be a minimal genus rational Seifert surface for $K$. The following two statements hold:
     \begin{itemize}
          \item[(1)] $c(\xi_K) \ne 0$, where $c(\xi_K)$ is the Heegaard Floer contact invariant associated to the contact structure $\xi_K$ coming from the open book of $(Y, K)$.
          \item[(2)] There exists a $\textup{Spin}^c$ structure $\mathfrak{s}$ on $Y$ such that $K$ satisfies $\displaystyle \tau(Y,K,\mathfrak s)= \frac{y([F])}{|[\mu]\cdot[\partial F]|}$.
     \end{itemize}
\end{thm}
Indeed, it follows that for a fibered, Floer simple knot $K$ in a rational homology sphere $Y$, the two conclusions of the theorem are equivalent.

\subsection{Notation}\label{sec:Notation}
We fix some notation that will be used throughout the paper. The singular homology and cohomology groups are all taken over the ring of integers $\mathbb Z$, unless a different coefficient ring is specified. Unless noted otherwise, $Y$ denotes a rational homology sphere. We let $K$ be an oriented knot in $Y$, and $M = Y \setminus \nu^{\circ}(K)$. We choose an oriented longitude $\lambda \in H_1(\partial M)$ whose orientation is coherent with the orientation of $K$. Let $\mu \in H_1(\partial M)$ be a meridian of $K$ with the property that $\mu \cdot \lambda = 1$ with respect to the orientation on $\partial M$ induced by $\partial \nu (K)$. Let $Y_n$ denote the manifold obtained by Dehn filling $M$ along the curve $n \cdot \mu + \lambda$. In particular, $Y_0$ denotes the filling of $M$ (surgery on $K$) along $\lambda$. The null slope of $K \subset Y$ is denoted $\alpha$, and that the surgery on $K$ with slope $\alpha$ is denoted $Y_\alpha$. Lastly, we often use the terms ``longitude" and ``framing": both refer to a slope at distance one from the meridian $\mu$.

\subsection{Organization}\label{sec:Organization} The rest of the paper is organized as follows. Section~\ref{section:Background} provides background from Heegaard Floer homology. Section~\ref{TightnessGeneralization} proves Proposition~\ref{prop:TightnessGeneralization}. Section~\ref{sec:S1S2} proves Theorem~\ref{thm:S1S2fiber} and Proposition~\ref{prop:FloerSimple}. Section~\ref{Thurston-norm} is devoted to some preliminary lemmas, followed by the proof of Theorems~\ref{SurfaceBundle} and~\ref{NormMinimizing}. The final section addresses potential directions for future research.

\subsection*{Acknowledgements} We are grateful to Kenneth Baker for pointing out Remark~\ref{S1S2Overtwisted} to us, to Matthew Hedden for his input to Proposition~\ref{prop:TightnessGeneralization}, to Tye Lidman for helpful conversations, and to Jacob Rasmussen for pointing out a mistake in an earlier draft. We thank the referee for valuable remarks and a thoughtful review. Y.~N. was partially supported by NSF grant
numbers DMS-1103976, DMS-1252992, and an Alfred P. Sloan Research Fellowship; F.~V. was partially supported by an NSF Simons travel grant.


\section{Background}\label{section:Background}

In this section we provide the Heegaard Floer homology background en route to proving the main results of the paper.

\subsection{Knot Floer homology}\label{subsec:HFK}

The primary goal of this subsection is to recall the construction of  knot Floer homology. We start by briefly reviewing the construction of a doubly pointed Heegaard diagram for a knot $K$ in a closed three-manifold $Y$ \cite{Ozsvath2004,Ozsvath2010}. Throughout the subsection, we mainly use the notation of \cite{Ozsvath2010}.

Let $(\Sigma, \mbox{\boldmath${\alpha}$},\mbox{\boldmath$\beta$}, w, z)$ be a doubly pointed Heegaard diagram for $K \subset Y$, in the following sense. Here, $\Sigma$ is an oriented surface of genus $g$, $\mbox{\boldmath${\alpha}$} = \left \{\alpha_1, \cdot \cdot \cdot , \alpha_g\right \}$ is a $g$-tuple of homologically linearly independent, pairwise disjoint, simple closed curves in $\Sigma$, so is $\mbox{\boldmath${\beta}$} = \left \{\beta_1, \cdot \cdot \cdot , \beta_g\right \}$. The two points $w$ and $z$ lie on $$\Sigma -\alpha_1 -\cdot \cdot \cdot -\alpha_g -\beta_1 -\cdot \cdot \cdot -\beta_g.$$ The curves $\mbox{\boldmath$\alpha$}$ and $\mbox{\boldmath$\beta$}$ specify a pair of handlebodies $U_{\mbox{\boldmath$\alpha$}}$ and $U_{\mbox{\boldmath$\beta$}}$ with common boundary $\Sigma$. We require that $(\Sigma, \mbox{\boldmath${\alpha}$},\mbox{\boldmath$\beta$},  w)$ is a Heegaard diagram for $Y$, and also that the knot $K$ is the union of two arcs $K_{\alpha},K_{\beta}$, where $K_{\alpha}\subset U_{\mbox{\boldmath$\alpha$}}$ is an unknotted arc connecting $z$ to $w$ and is disjoint from the disks attached to $\alpha_1,\dots,\alpha_g$, and $K_{\beta}\subset U_{\mbox{\boldmath$\beta$}}$ is an unknotted arc connecting $w$ to $z$ and is disjoint from the disks attached to $\beta_1,\dots,\beta_g$.

Spin$^c$ structures on $Y$ can be seen as {\it homology classes} of non-vanishing
vector fields forming an affine space over $H^2(Y)$. Two nowhere vanishing vector fields on $Y$ are homologous if they are homotopic on the complement of a ball embedded in $Y$. From the combinatorics of the Heegaard diagram one can construct a function
$$\mathfrak s_w : \mathbb T_{\mbox{\boldmath{$\alpha$}}} \cap \mathbb T_{\mbox{\boldmath{$\beta$}}} \to \mathrm{Spin}^c(Y),$$
where $\mathbb T_{\mbox{\boldmath{$\alpha$}}}$ and $\mathbb T_{\mbox{\boldmath{$\beta$}}}$ are
two totally real half-dimensional tori in the symmetric product $\text{Sym}^g(\Sigma)$ which is endowed with an almost
complex structure. The map $\mathfrak s_w$ sends an intersection point $\mathbf x$ to the homology class of a vector field. There is also a relative version $\underline{\mathrm{Spin}^c} (Y, K)$. It consists of homology classes of vector fields on the knot complement $M$ which point outwards at the boundary; one has an analogous map
\[
\underline{\mathfrak s}_{w,z} : \mathbb T_{\mbox{\boldmath{$\alpha$}}} \cap \mathbb T_{\mbox{\boldmath{$\beta$}}} \to \underline{\mathrm{Spin}^c}(Y, K).
\]
There is another equivalent definition of relative Spin$^c$ structure in the literature \cite{OSzLink}, where the boundary condition is that the vector field on $\partial M$ is the (up to isotopy) canonical vector field tangent to $\partial M$. Let $\xi\in \underline{Spin^c}(Y, K)$ be represented by the homology class of a vector field $v$. The Spin$^c$ structure $[-v]$, denoted $J(\xi)$, is called {\it the conjugate} of $\xi$. It is clear that
\begin{equation}\label{conjugate}
c_1(J(\xi))= -c_1(\xi).
\end{equation}
Equivalently, a relative Spin$^c$ structure on $(Y, K)$ is a nowhere vanishing vector field on $Y$ that contains $K$ as a closed orbit. Similar to the closed case, $\underline{\mathrm{Spin}^c}(Y, K)$ is an affine space over $H^2(Y, K)$. There is a natural map
\[
G_{Y, K}: \underline{\mathrm{Spin}^c}(Y, K) \to \mathrm{Spin}^c(Y)
\]
which is equivariant with respect to the action by $H^2(Y, K)$. That is, letting
\[
\iota:H^2(Y, K) \to H^2(Y)
\]
be the map induced by the inclusion, we have for each $a\in H^2(Y,K)$
\begin{equation}\label{eq:Equivariant}
G_{Y, K}(\xi+a) = G_{Y, K}(\xi)+\iota(a).
\end{equation}

Given a doubly pointed Heegaard diagram $(\Sigma, \mbox{\boldmath${\alpha}$},\mbox{\boldmath$\beta$},w, z)$ which represents a rationally null-homologous knot $K \subset Y$ and $\xi \in \underline{\text{Spin}^c}(Y, K)$, Ozsv\'{a}th and Szab\'{o} construct a $(\mathbb Z \oplus \mathbb Z)$--filtered chain complex $CFK^{\infty}(Y,K, \xi)$.
The generating set is the subset $\mathfrak{T}(\xi) \subset \mathbb T_{\mbox{\boldmath{$\alpha$}}} \cap \mathbb T_{\mbox{\boldmath{$\beta$}}} \times \mathbb Z \times \mathbb Z$ consisting of all elements $[\mathbf x, i, j]$ with the property that
\begin{equation}\label{SpinStrFormula}
\underline{\mathfrak s}_{w,z}(\mathbf x)+(i-j)PD[\mu]= \xi.
\end{equation}
 The differential counts certain pseudo-holomorphic disks connecting the generators with the boundary mapping to $\mathbb T_{\mbox{\boldmath{$\alpha$}}} \cup \mathbb T_{\mbox{\boldmath{$\beta$}}}$. The two basepoints $w$ and $z$ give rise to codimension $2$ submanifolds $\left \{w \right \} \times\mathrm{ Sym}^{g-1}(\Sigma)$, respectively $\left \{z\right \} \times \mathrm{Sym}^{g-1}(\Sigma)$ of $\mathrm{Sym}^g(\Sigma)$. More precisely, the chain complex is endowed with the differential
\[
\displaystyle \partial^{\infty}[\mathbf x,i,j]=\sum_{\mathbf y\in\mathbb T_{\mbox{\boldmath{$\alpha$}}} \cap \mathbb T_{\mbox{\boldmath{$\beta$}}}}\sum_{\{\phi\in\pi_2(\mathbf x,\mathbf y)|\mu(\phi)=1\}}\#(\widehat{\mathcal M}(\phi))[\mathbf
y,i-n_{w}(\phi), j-n_{z}(\phi)]
\]
where $\pi_2(\mathbf x, \mathbf y)$ denotes the set of homotopy classes of Whitney disks connecting $\mathbf x$ and $\mathbf y$, $\mu (\phi)$ is the
Maslov index of $\phi$, $\#(\widehat{\mathcal M}(\phi))$ is the count of holomorphic representatives of $\phi$, $n_w(\phi) = \# \phi \cap \left \{w\right \} \times \text{Sym}^{g-1}(\Sigma)$, and similarly for $n_z(\phi)$.
If $[\mathbf x, i, j]\in \mathfrak{T}(\xi)$ and $\phi\in \pi_2(\mathbf x, \mathbf y)$, then $[\mathbf
y,i-n_{w}(\phi), j-n_{z}(\phi)] \in \mathfrak{T}(\xi)$. The map
\[
\mathcal F: \mathfrak{T}(\xi)\to \mathbb Z \oplus\mathbb Z,
\]
where $\mathcal F([\mathbf x, i, j])= (i,j)$ induces a $\mathbb Z\oplus\mathbb Z$ filtration on $CFK^{\infty}(Y, K, \xi)$. Although, by construction, the chain complex depends on the choice of a doubly pointed Heegaard diagram and also a representative of $\xi$, Ozsv\'ath and Szab\'o
proved that its filtered chain homotopy type is an invariant of the triple $(Y, K, \xi)$, as the notation suggests. Let $\widehat{CFK}(Y, K, \xi)$ be the sub-quotient complex of $CFK^{\infty}(Y, K, \xi)$ with $i=j=0$, endowed with the induced differential $\widehat \partial$.
Its homology, denoted $\widehat{HFK}(Y, K, \xi)$, is trivial for all but finitely many $\xi \in \underline{\textup{Spin}^c}(Y, K)$.

The knot Floer homology $\widehat{HFK}(Y, K)$ is a finitely generated abelian group (with an absolute grading) that  decomposes as a direct sum
\[
\widehat{HFK}(Y, K) \cong \bigoplus_{\xi \in \underline{\mathrm{Spin}^c}(Y, K)}\widehat{HFK}(Y, K, \xi).
\]

\subsection{The rational-valued $\tau$ invariant}\label{sec:TauDef} In this subsection we define the rational-valued $\tau$ invariant associated to a knot $K$ in a rational homology sphere $Y$. Suppose that $F$ is a rational Seifert surface for $K$. As in Subection~\ref{subsec:HFK}, let $(\Sigma, \mbox{\boldmath${\alpha}$},\mbox{\boldmath$\beta$}, w, z)$ be a doubly pointed Heegaard diagram for $(Y, K)$, $\mathfrak s\in \mathrm{Spin}^c(Y)$, $\xi\in\underline{\mathrm{Spin}^c}(Y,K)$.
Set $$\mathcal B_{Y,K}=\left\{\xi\in\underline{\mathrm{Spin}^c}(Y,K)\left|\:\widehat{HFK}(Y,K,\xi)\ne0\right.\right\}.$$
There exists a unique affine map $$A : \underline{\mathrm{Spin}^c}(Y,K) \to \mathbb{Q}$$ satisfying
\begin{equation}\label{eq:AlexanderDifference}
A(\xi_1)-A(\xi_2)=\frac{\langle \xi_1-\xi_2,[F, \partial F]\rangle}{|[\partial{F}]\cdot[\mu]|},
\end{equation}
and
\begin{equation}\label{eq:Aminmax}
\max\{A(\xi)|\:\xi\in\mathcal B_{Y,K}\}=-\min\{A(\xi)|\:\xi\in\mathcal B_{Y,K}\}.
\end{equation}
In fact, we can define $A$ as
\begin{equation}\label{AlexanderGrading}
A(\xi)=\frac{\langle c_1(\xi),[F, \partial F]\rangle- [\mu] \cdot [\partial F]}{2[\mu] \cdot [\partial F]}.
\end{equation}
We refer to $A$ as the \emph{Alexander grading}. Note that $A$ does not depend on the choice of a rational Seifert surface. The Alexander grading gives rise to a filtration $\mathcal F$ on $\widehat{CF}(Y)$ in the standard way, i.e. we let
$$\mathcal F(Y,K,m)=\bigoplus_{\{\xi\in\underline{\mathrm{Spin}^c}(Y,K),\:A(\xi)\le m\}}\widehat{CFK}(Y,K,\xi),\quad m\in\mathbb Q.$$
Positivity of intersections of $J$-holomorphic Whitney disks with the hypersurfaces determined by $z$ and $w$ ensures that $\mathcal F(m)$ is a subcomplex; that is, $\widehat \partial \mathcal F (m) \subset \mathcal F(m)$ and hence $\mathcal F$ defines a filtration. We have the following finite sequence of inclusions
\begin{equation}\label{eq:filtration}
0=F(Y,K,-j) \hookrightarrow F(Y,K,-j+1) \hookrightarrow \cdots F(Y,K,n)=\widehat{CF}(Y),
\end{equation}
where the finiteness of the sequence follows from the fact the number of intersection points $\mathbf x \in T_{\mbox{\boldmath{$\alpha$}}} \cap \mathbb T_{\mbox{\boldmath{$\beta$}}}$ is finite.
Let $\iota_m : \mathcal F(Y,K,m)\to \widehat{CF}(Y)$ be the inclusion map, and let $(\iota_m)_*$ be the induced map on homology.

Following \cite{HeddenTau}, we make the following definition of the rational-valued $\tau$ invariant.

\begin{defn} \label{def:RationalTau}
Using the notation of Subsection~\ref{sec:TauDef}, let $K$ be a knot in a rational homology sphere $Y$, endowed with a rational Seifert surface $F$. Given $\mathfrak s\in\spin(Y)$ and $\mathbf a\in \widehat{HF}(Y,\mathfrak s)$, define
$$\tau_{\mathbf a}(Y,K)=\min\big\{m\in\mathbb Q\:\big|\:\mathbf a\in\mathrm{im}\:(\iota_m)_*\big \},$$
and
$$\tau(Y,K,\mathfrak s)=\min\big\{m\in\mathbb Q\:\big|\:\widehat{HF}(Y,\mathfrak s)\subset\mathrm{im}\:(\iota_m)_*\big \}.$$
\end{defn}
Note that the minimum is actually attained: see~\eqref{eq:filtration}. It is straightforward to check that when $Y = S^3$, $\tau(Y,K,\mathfrak s)$ agrees with the integer-valued $\tau(K)$ (defined in \cite{Ozsvath2003}).

{\rmk \label{rmk:TauProperties} When $Y = S^3$, it is known that $\tau(K)$ gives a lower bound on the four-ball genus \cite[Corollary~1.3]{Ozsvath2003}. Raoux, in~\cite{Raoux}, has given a slightly different definition of $\tau(Y,K,\mathfrak s)$: she has studied various properties of the rational valued invariant; in particular, she proves a generalization of the genus bound result.
 }

\subsection{Heegaard Floer homology of large surgeries, and a relevant exact sequence}\label{subsect:Triangle}

We start by reviewing the ``large surgery formula" for a rationally null-homologous knot in $Y$. For a more detailed discussion, see \cite{Ozsvath2010}. Let $K\subset Y$ be an oriented knot endowed with a framing $\lambda$. Let $[K]$, as an element of $H_1(Y)$, be of order $p$. For a fixed $\xi \in \underline{\mathrm{Spin}^c}(Y, K)$, let $C_{\xi}$ be the chain complex $CFK^{\infty}(Y, K, \xi)$. There are two projection maps
\begin{equation}\label{projection}
C_{\xi}\{i\ge 0 \text{ or } j\ge 0\} \xrightarrow{v^+_{\xi}} C_{\xi}\{i\ge 0\} \text{   and   } C_{\xi}\{i\ge 0 \text{ or } j\ge 0\} \xrightarrow{h^+_{\xi}} C_{\xi}\{j\ge 0\}.
\end{equation}
Denote
\[
A^+_{\xi}(Y, K)= C_{\xi}\{i\ge 0 \text{ or } j\ge 0\} \text{ and }B^+_{\xi}(Y, K) = CF^+(Y, G_{Y, K}(\xi))
\]
and use the identifications
\[
C_{\xi}(i\ge 0)\cong CF^+(Y, G_{Y, K}(\xi)),
\]
\begin{equation}\label{eq:identification}
C_{\xi}(j\ge 0)\cong CF^+(Y, G_{Y, -K}(\xi))\cong CF^+(Y, G_{Y, K}(\xi+PD[\lambda])),
\end{equation}
where here $\lambda$ is thought as the push-off of $K$ inside $Y$ using the framing $\lambda$. Then the canonical projection maps of~\eqref{projection} may be written as
\[
v^+_{\xi}: A^+_{\xi}(Y, K) \to B^+_{\xi}(Y, K) \text{ and } h^+_{\xi}: A^+_{\xi}(Y, K) \to B^+_{\xi+PD[\lambda]}(Y, K).
\]
Since the $(\mathbb Z\oplus\mathbb Z)$--filtered chain homotopy type of $C_{\xi}$ is an invariant of the triple $(Y, K, \xi)$, the chain homotopy classes of the maps $v^+_{\xi},h^+_{\xi}$ are also invariants of the triple $(Y, K, \xi)$.

Let $W'_n(K)$ be the cobordism obtained from turning around the two-handle cobordism
from $-Y$ to $- Y_n$ (see Section~\ref{sec:Notation} for the definition of $Y_n$). It is easy to verify that
\[
H_2(W'_n(K)) \cong \mathbb Z,
\]
where the generator is the class of the capped off rational Seifert surface in $W'_n(K)$.
As in \cite[Proposition~2.2]{Ozsvath2010}, there is a well-defined map
\[
E_{Y, n , K}: \mathrm{Spin}^c(W'_n(K))\to \underline{\mathrm{Spin}^c}(Y, K),
\]
that restricts a Spin$^c$ structure on the four-manifold to the knot complement. We point out that $E_{Y,n, K}$ depends on the choice of $\lambda$, a longitude for $K$. We remind the reader that we chose a longitude for $K$ in the beginning of the subsection.
Note that
\[
H_2(W'_n(K), Y) \cong \mathbb Z
\]
is generated by $[S]$, where $S$ is the core of the two-handle attached to $Y$ in the cobordism $W'_n(K)$. We orient $S$ so that its boundary orientation is coherent with the orientation of $K$.
When $n$ is sufficiently large, the two-handle cobordism is a negative definite four-manifold, and therefore, the self-intersection number of $S$ is negative.

The following theorem relates the Heegaard Floer complex of large surgeries on $K\subset Y$ to the knot Floer complex associated to $(Y, K)$.
{\thm \cite[Theorem~4.1]{Ozsvath2010}\label{LargeSurgery} Let $K \subset Y$ be a rationally null-homologous knot in a closed, oriented
three-manifold, equipped with a framing $\lambda$. Then, for all sufficiently large $n$, there is a
map
\[
\Xi: \mathrm{Spin}^c(Y_n) \to \underline{\mathrm{Spin}^c}(Y, K)
\]
with the property that for all $\mathfrak t \in \mathrm{Spin}^c(Y_n)$, the chain complex $CF^+(Y_n,\mathfrak t)$ is
represented by the chain complex
\[
A^+_{\Xi(\mathfrak t)}= C_{\Xi (\mathfrak t)}\{i\ge 0 \text{ or }j\ge 0\}
\]
in the sense that there are isomorphisms
\[
\Psi^+_{\mathfrak t, n}:CF^+(Y_n, \mathfrak t) \to A^+_{\Xi (\mathfrak t)}(Y, K).
\]
Furthermore, fix $\mathfrak t \in \mathrm{Spin}^c(Y_n)$, and let $\Xi(\mathfrak t)= \xi$. There are Spin$^c$ structures
$\mathfrak x=\mathfrak x(\mathfrak t), \mathfrak y=\mathfrak y(\mathfrak t)\in \mathrm{Spin}^c(W'_n(K))$ with $E_{Y, n ,K }(\mathfrak x) = \xi$, and $\mathfrak y = \mathfrak x + PD[S]$ with the property that the maps $v^+_{\xi}$ and $h^+_{\xi}$ correspond to the maps induced by the cobordism $W'_n(K)$ equipped with $\mathfrak x$ and $\mathfrak y$, respectively. }
\\

Throughout the proof of Theorem~\ref{SurfaceBundle} we use a surgery exact triangle relating the Floer homologies of $Y$, $Y_0$, and $Y_n$. Before stating the sequence we make some notational conventions. Fix $\mathfrak t \in \mathrm{Spin}^c(Y_n)$. We define
\begin{equation}\label{SpinEquivalence}
[\mathfrak t]_{Y_n} = \left\{\mathfrak t' \in \mathrm{Spin}^c(Y_n) |\mathfrak t' - \mathfrak{t} \in \langle PD[\lambda]\rangle\right\},
\end{equation}
where $\langle PD[\lambda]\rangle$ denotes the cyclic group generated by $PD[\lambda]\in H^2(Y_n)$. Correspondingly, we define
\[
HF^+(Y_n, [\mathfrak t]_{Y_n}) = \bigoplus_{\mathfrak t' \in [\mathfrak t]_{Y_n}}HF^+(Y_n, \mathfrak t').
\]
For $\mathfrak s \in \mathrm{Spin}^c(Y)$, we define $[\mathfrak s]_{Y}$ and $HF^+(Y, [\mathfrak s]_{Y})$, similarly. Note that Spin$^c$ structures on $Y$ which are cobordant to a fixed Spin$^c$ structure on $Y_n$ form an affine space over the image of
\[
H^2(W'_n,Y_n) \to H^2(Y).
\]
It is straightforward to check that this image is $\langle PD[\lambda] \rangle \in H^2(Y)$. Therefore, there exists a unique orbit $[\mathfrak s_{\mathfrak t}]_{Y}$ which is cobordant to $[\mathfrak t]_{Y_n}$, for some $\mathfrak s_{\mathfrak t} \in \mathrm{Spin}^c(Y)$, in $W'_n$.
{\thm\cite[Theorem~3.3.3]{Cebanu2012}\label{ExactTriangle}  Let $Y$ be a closed, oriented three-manifold, and $K\subset Y$ be a rationally
null-homologous knot endowed with a framing $\lambda$. There is a map $$Q: \mathrm{Spin}^c(Y_{0}) \to \mathrm{Spin}^c(Y_n)/\langle PD[\lambda]\rangle$$ such that for a positive integer $n$ and $\mathfrak t \in \mathrm{Spin}^c(Y_n)$, there is a long
exact sequence
\begin{equation}\label{SurgeryTriangle}
\xymatrix{
HF^+(Y, [\mathfrak s_{\mathfrak t}]_{Y}) \ar[rr] &&HF^+(Y_{0}, Q^{-1}([\mathfrak t]_{Y_n})) \ar[ld]\\
&HF^+(Y_n, [\mathfrak t]_{Y_n}) \ar[lu]^{F}& .
}
\end{equation}
Here, $[\mathfrak s_{\mathfrak t}]_{Y}$ is the unique orbit of Spin$^c$ structures on $Y$ that is cobordant to $[\mathfrak t]_{Y_n}$ in $W'_n$.\footnote{In the statement of Theorem~\ref{ExactTriangle}, $Y_0$ denotes the Dehn surgery on $K$ with slope $\lambda$. See Subsection~\ref{sec:Notation}.}
}
\\

Theorem~\ref{ExactTriangle} is a generalization of \cite[Theorem~9.19]{Ozsvath2004a} with an almost identical proof. In \cite[Theorem~9.19]{Ozsvath2004a}, the knot is assumed to be null-homologous. The proof starts with constructing a multi Heegaard diagram $(\Sigma, \mbox{\boldmath${\alpha}$},\mbox{\boldmath$\beta$}, \mbox{\boldmath$\gamma$}, \mbox{\boldmath$\delta$},w)$ with $\Sigma$ a surface of genus $g$ where $(\Sigma, \mbox{\boldmath${\alpha}$},\mbox{\boldmath$\beta$}, w)$, $(\Sigma, \mbox{\boldmath${\alpha}$},\mbox{\boldmath$\gamma$}, w)$, and $(\Sigma, \mbox{\boldmath$\alpha$}, \mbox{\boldmath$\delta$}, w)$ describe $Y$, $Y_0$, and $Y_n$, respectively. Then appropriate maps will be defined to get the exact sequence as desired. In our case there will be $\langle [\lambda]\rangle$--orbits of Spin$^c$ structures in the statement since the knot is not null-homologous. Also, the proof of \cite[Theorem~9.19]{Ozsvath2004a} needs to be modified when we define the map $Q$. Let $X$ be the four-manifold cobordism, specified by $(\Sigma, \mbox{\boldmath${\alpha}$}, \mbox{\boldmath$\gamma$}, \mbox{\boldmath$\delta$}, w)$. For a given $\mathfrak s \in \mathrm{Spin}^c(Y_0)$, there is a unique orbit $[\mathfrak t_{\mathfrak s}]_{Y_n}$, such that there is a Spin$^c$ structure $\mathfrak s_{\alpha, \gamma, \delta} \in \mathrm{Spin}^c(X)$ with $\mathfrak s_{\alpha, \gamma, \delta}|_{Y_0}= \mathfrak s$, $\mathfrak s_{\alpha, \gamma, \delta}|_{Y_n} \in [\mathfrak t_{\mathfrak s}]_{Y_n}$. In other words, fixing $\mathfrak s\in \mathrm{Spin}^c(Y_0)$, there is a $\mathfrak t \in \mathrm{Spin}^c(Y_n)$ with the property that there is a unique $\mathfrak s_{\alpha, \gamma, \delta} \in \mathrm{Spin}^c(X)$ that extends $\mathfrak t$, some unique Spin$^c$ structure on the manifold specified by $(\Sigma, \mbox{\boldmath$\gamma$}, \mbox{\boldmath$\delta$}, w)$, and any element of the orbit $[\mathfrak t_{\mathfrak s}]_{Y_n}$. This describes the map $Q$ in the theorem.

In what follows, we will define $F$, the map relating $Y_n$ and $Y$. We will skip the definition of the other two maps in the exact sequence, and instead refer the reader to \cite[Theorem~3.3.3]{Cebanu2012} and \cite[Theorem~9.19]{Ozsvath2004a}.

Heegaard Floer homology is functorial with respect to cobordisms. Indeed, if $W$ is a smooth, connected, oriented cobordism with $\partial W = -Y_1 \cup Y_2$ which is equipped with a Spin$^c$ structure $\mathfrak s$ with restriction $\mathfrak t_i = \mathfrak s|_{Y_i}$ for $i=1,2$, then there is an induced chain map
\[
f^+_{W, \mathfrak s}: CF^+(Y_1, \mathfrak t_1) \to CF^+(Y_2, \mathfrak t_2).
\]
The construction of $f^+_{W, \mathfrak s}$ uses some auxiliary data like a Heegaard triple and an almost complex structure on $\mathrm{Sym}^g(\Sigma)$, but
the chain homotopy type of $f^+_{W, \mathfrak s}$ is an invariant of the pair $(W, \mathfrak s)$.
If $\mathfrak t_1,\mathfrak t_2$  have torsion first Chern classes, $f^+_{W, \mathfrak s}$
 is homogeneous of degree
\begin{equation}\label{degreeshift}
\displaystyle \frac{c_1(\mathfrak s)^2-2\chi(W) -3 \sigma(W)}{4},
\end{equation}
where $\chi$ and $\sigma$ denote the Euler characteristic and the signature of the four-manifold $W$, respectively.

The map $F$ in (\ref{SurgeryTriangle}) is induced by
\begin{equation}\label{eq:F3}
f=\sum_{\mathfrak s\in\mathrm{Spin}^c(W_n'(K)), \mathfrak s|_{Y_n}\in[\mathfrak t]_{Y_n}}f^+_{W_n'(K),\mathfrak s}.
\end{equation}

\subsection{The evaluation of the first Chern class}\label{sec:FirstChernClassFormula}

A key step in the proof of Theorems~\ref{SurfaceBundle},~\ref{NormMinimizing} and Proposition~\ref{prop:TightnessGeneralization} is the evaluation of the first Chern class of a Spin$^c$ structure on a second homology class. Such an evaluation is often not that straightforward to compute, however, in certain cases it is fairly well understood. Let $K$ be an oriented rationally null-homologous knot in a closed three-manifold $Y$, endowed with a framing $\lambda$ and a rational Seifert surface $F$. Let also $p$ be the order of $[K]\in H_1(Y)$. We start by stating a lemma that studies the evaluation of the first Chern class of a relative Spin$^c$ structure with either the lowest or the highest Alexander grading, evaluated on the homology class $[F, \partial F]$. Recall that $\mathcal B_{Y,K}$ is the set of all relative Spin$^c$ structures in which the knot Floer homology is not zero.
{\lemma\label{ChernClassFormula}\cite[Proposition~6.4]{YiNi2014}
Let $K$ be an oriented rationally null-homologous knot in a closed three-manifold $Y$. Let also $F$ be a minimal genus rational Seifert surface for $K$. Suppose that $K$, as an element of $H_1(Y)$, has order $p$. Then
\[
\min_{ \xi \in \mathcal B_{Y,K} }\langle c_1 (\xi), [F, \partial F] \rangle = \chi(F),
\]
\[
\max_{\xi\in \mathcal B_{Y,K} }\langle c_1 (\xi), [F, \partial F] \rangle =-\chi(F)+2p.
\]
}
\begin{proof}
This is a direct consequence of \cite[Theorem~1.1]{Ni2009}, where it is proved that
\[
\max\{\langle c_1(\xi)-PD[\mu],[F, \partial F]\rangle|\:\xi\in\mathcal B_{Y,K}\}= -\chi(F) + |F\cdot [\mu]|.
\]
This, together with Equation~\eqref{eq:Aminmax} will give us the result.
\end{proof}

The next lemma studies the evaluation of the first Chern class of a specific Spin$^c$ structure on the two-handle cobordism $W'_n(K)$, for some positive integer $n$, evaluated on the capped off Seifert surface. This will be of use in the proof of Theorem~\ref{SurfaceBundle}.

\begin{lemma}\label{lem:c1Evaluation}
Let $Y$ be a rational homology sphere, $K\subset Y$ be a knot of order $p$ in $H_1(Y)$, and $F$ be a rational Seifert surface for $K$ such that $[F,\partial F]$ represents the generator of $H_2(Y,K)$. Suppose that the null slope of $K$ is a framing. Let $\widetilde F\subset W'_n(K)$, for some positive integer $n$, be the closed surface obtained by capping off $\partial F$ with disks. Let ${\xi}\in\underline{\mathrm{Spin}^c}(Y,K)$ be a relative Spin$^c$ structure, and $\mathfrak x\in \mathrm{Spin}^c(W'_n(K))$ be a Spin$^c$ structure with $E_{Y, n , K}(\mathfrak x)={\xi}$. Then
\[
\langle c_1(\mathfrak x),[\widetilde F]\rangle=\langle c_1(\xi),[F, \partial F]\rangle-p(n+1).
\]
\end{lemma}
\begin{proof}
Let $H\subset W'_n(K)$ be the two-handle attached to $Y\times[0,1]$. The natural map
\[
\iota^*: H^2(W'_n(K),H) \to H^2(W'_n(K)),
\]
induced by inclusion, is an isomorphism.
Since $H^2(W'_n(K),H)\cong H^2(Y,K)$, the inverse of $\iota^*$ gives a map
\[
\varepsilon: H^2(W'_n(K))\to  H^2(Y,K).
\]
The affine map $E=E_{Y, n , K}: \mathrm{Spin}^c(W'_n(K))\to \underline{\mathrm{Spin}^c}(Y,K)$ is modeled on $\varepsilon$. That is, $E$ covers $\varepsilon$ as a torsor map.
Fix any $\xi_0\in \underline{\mathrm{Spin}^c}(Y,K)$, let $\mathfrak x_0=E^{-1}(\xi_0)$, and let
\begin{equation}\label{eq:DiffC}
C=\langle c_1(\mathfrak x_0),[\widetilde F]\rangle-\langle c_1(\xi_0),[F, \partial F]\rangle.
\end{equation}
Assume $a=\mathfrak x-\mathfrak x_0\in H^2(W'_n(K))$, then $\varepsilon(a)=\xi-\xi_0$, and one has
\[
\langle a,[\widetilde F]\rangle=\langle \varepsilon(a),[F,\partial F]\rangle.
\]
It follows that
\begin{eqnarray}
\langle c_1(\mathfrak x),[\widetilde F]\rangle&=&\langle c_1(\mathfrak x_0)+2a,[\widetilde F]\rangle\nonumber\\
&=&\langle c_1(\xi_0)+2\varepsilon(a),[F, \partial F]\rangle+C\nonumber\\
&=&\langle c_1(\xi),[F, \partial F]\rangle+C.\label{eq:+C}
\end{eqnarray}
Thus the constant $C$ does not depend on the choice of $\xi_0$.

Let $\xi_0$ be a relative Spin$^c$ structure satisfying $\langle c_1(\xi_0),[F,\partial F]\rangle=p$. Let $\eta_0=J(\xi_0)+PD[\mu]$, where $J$ is the conjugation on $\underline{\mathrm{Spin}^c}(Y,K)$. Then, using~\eqref{conjugate},
\[
\langle c_1(\eta_0),[F, \partial F]\rangle=\langle -c_1(\xi_0)+2PD[\mu],[F,\partial F]\rangle=-p+2p=p=\langle c_1(\xi_0),[F,\partial F]\rangle.
\]
Hence, it follows from~\eqref{eq:+C} that
\[
\langle c_1(E^{-1}(\eta_0)),[\widetilde F]\rangle=\langle c_1(\mathfrak x_0),[\widetilde F]\rangle.
\]
Since $H^2(W'_n(K);\mathbb Q)\cong\mathbb Q$, the square of any $a\in H^2(W'_n(K))$ determines and is determined by $|\langle a,[\widetilde F]\rangle|$. So, using (\ref{degreeshift}), we conclude that the degree of $h_{\eta_0}$ is equal to the degree of $h_{\xi_0}$.

On the other hand, it is well known that $h_{\eta_0}$ is chain homotopy equivalent to $v_{\xi_0}$ after interchanging the roles of $i,j$. (See the proof of \cite[Proposition~8.2]{OSzLink}.) So the degree of $v_{\xi_0}$ is equal to the degree of $h_{\eta_0}$. It follows that the degrees of $v_{\xi_0}$ and $h_{\xi_0}$ are equal. By Theorem~\ref{LargeSurgery}, $v_{\xi_0}$ is induced by $\mathfrak x_0$, and $h_{\xi_0}$ is induced by $\mathfrak x_0 + PD[S]$. By (\ref{degreeshift}), we have that $c_1^2(\mathfrak x_0)=c_1^2(\mathfrak x_0 + PD[S])$. We claim that
\begin{equation}\label{eq:c1Neg}
\langle c_1(\mathfrak x_0),[\widetilde F]\rangle=-\langle c_1(\mathfrak x_0+ PD[S]),[\widetilde F]\rangle.
\end{equation}
Note that the quality of the $c_1^2$ will give us that
\[
\langle c_1(\mathfrak x_0),[\widetilde F]\rangle=\pm\langle c_1(\mathfrak x_0+ PD[S]),[\widetilde F]\rangle.
\]
The plus sign would imply that $PD[S]=0$, a contradiction.
The orientations of both $\partial F$ and $\partial S$ are coherent with respect to the orientation of $K$. Therefore, using that the null slope of $K$ is a framing, $\widetilde F$ is obtained from $F$ by gluing $p$ copies of $-S$ along the boundary components. Since the framing of the two-handle in $W'_n$ is $-n$, and also that $S$ is the core of the two handle attached to $Y$ in $W'_n$, we have
\begin{equation}\label{eq:PD[S]}
\langle PD[S], [\widetilde F]\rangle=[S]\cdot[\widetilde F]=p[S]\cdot[-S]=np.
\end{equation}
Thus,  \eqref{eq:c1Neg} implies that $\langle c_1(\mathfrak x_0),[\widetilde F]\rangle=-np$. Now, using (\ref{eq:DiffC}), we get that $C=-(n+1)p$. This, together with~\eqref{eq:+C} will give us the result.
\end{proof}

\begin{cor}\label{cor:DiffTors}
Using the assumptions of Lemma~\ref{lem:c1Evaluation}, let $\xi_i\in\underline{\mathrm{Spin}^c}(Y,K)$ be a relative Spin$^c$ structures satisfying
\[
\langle c_1(\xi_1),[F, \partial F]
\rangle=\langle c_1(\xi_2),[F,\partial F]
\rangle ,
\]
 and $\mathfrak x_i\in \text{Spin}^c(W'_n(K))$ be Spin$^c$ structures with $E_{Y, n , K}(\mathfrak x_i)=\xi_i$, $i=1,2$. Then $\mathfrak x_1-\mathfrak x_2$ is a torsion element in $H^2(W'_n(K))$.
\end{cor}
\begin{proof}
By Lemma~\ref{lem:c1Evaluation},
\[
2\langle\mathfrak x_1-\mathfrak x_2,[\widetilde F]\rangle=\langle c_1(\mathfrak x_1)-c_1(\mathfrak x_2),[\widetilde F]\rangle=0.
\]
Since $H^2(W'_n(K);\mathbb Q)=\mathbb Q$,  it follows that $\mathfrak x_1-\mathfrak x_2$ is torsion.
\end{proof}

\subsection{Heegaard Floer contact invariant associated to fibered knots}\label{sec:Contact}

This subsection is devoted to defining the Heegaard Floer contact invariant associated to a fibered knot $K$ in a closed three-manifold $Y$. We choose $\mathbb F = \mathbb Z/2$ as the coefficient ring for the Heegaard Floer homology, to avoid any sign ambiguities. We do not review many of the concepts and definitions but instead refer the reader to~\cite{etnyre2006} for a review of
contact geometry, and to~\cite{Ozsvath2005} for the Heegaard Floer contact invariant in the case of fibered null-homologous knots. See~\cite{Hedden2011} for more details. A fibered knot $K\subset Y$ induces a rational open book decomposition and, therefore, a contact structure $\xi_K$~\cite{Baker2012}. Hedden and Plamenevskaya, in~\cite{Hedden2011}, studied the contact structure $\xi_K$ in terms of the knot Floer homology
of $K$. (See also~\cite{Ozsvath2005}.) More precisely, the ``bottommost" filtered subcomplex in the filtration of $\widehat{CF}(-Y)$ induced by $K$ has homology $\mathbb F$, that is, $H_*\mathcal (\mathcal F(\text{bottom}))\cong \mathbb F .\langle c \rangle $. If $\iota :  \mathcal F(\text{bottom}) \to \widehat{CF}(-Y)$ is the inclusion map of the lowest non-trivial subcomplex, the contact invariant associated to $K$ is defined to be $\iota_*(c)=c(\xi_K)\in \widehat{HF}(-Y)$. Here, $\widehat{HF}(-Y)$ is the Heegaard Floer homology of the manifold $Y$ with its orientation reversed and $\iota_*$ is the induced map in homology. (See \cite[Theorem~1.1]{Hedden2011}.) In words, the contact invariant of $\xi_K$ is the image of the generator of the homology of the bottom filtered subcomplex in the Floer homology of $Y$, under the natural inclusion induced map. When $K$ is fibered (instead of being only rationally fibered), the contact invariant just defined is simply Ozsv\'{a}th and Szab\'{o}'s definition~\cite[Definition~1.2]{Ozsvath2005}.


\section{Fibered, Floer simple knots induce tight contact structures}\label{TightnessGeneralization}

In this section we give a proof of Proposition~\ref{prop:TightnessGeneralization}.
Recall that the function $y(h)$ in the statement of Proposition~\ref{prop:TightnessGeneralization} is defined as
\[
y(h) = \max_{\left \{ \xi \in \underline{\textup{Spin}^c}(Y, K) : \widehat{HFK}(Y, K, \xi) \ncong 0 \right \}}\langle \mathfrak H (\xi), h \rangle,
\]
where $h \in H_2(Y, K; \mathbb Q)$, $\mathfrak H$ is the affine function $\mathfrak H : \underline{\textup{Spin}^c}(Y, K) \to H^2(Y, K; \mathbb Q)$ defined as
\[
\displaystyle \mathfrak H (\xi) = \frac{c_1(\xi) - PD[\mu]}{2}.
\]
Note that $y$ is a rational-valued function.

\begin{proof}[Proof of Proposition~\ref{prop:TightnessGeneralization}]
In order to show that the contact structure $\xi_K$ is tight, we will show that the Heegaard Floer contact invariant $c(\xi_K)$ is non-zero~\cite{Hedden2011}.
By \cite[Lemma~4.5]{Rasmussen2003} and its proof, there exists a unique filtered chain complex $C'$ such that $C'$ is filtered chain homotopy equivalent to $\widehat{CFK}(-Y,K)$, and $C'\cong \widehat{HFK}(-Y, K)$ as an abelian group. Here we use $\mathbb F=\mathbb Z/2\mathbb Z$ coefficients for the Heegaard Floer homology groups. Since $K$ is Floer simple, the differential on $C'$ is zero. Consequently, the inclusion map $\iota :  \mathcal F(\text{bottom}) \to \widehat{CF}(-Y)$ induces on the homology level the inclusion map of a nontrivial subgroup of $C'$ to $C'$.
Thus the contact invariant is non-zero. In particular, $\xi_{K}$ is tight.

Without loss of generality we may assume that the rational Seifert surface $F$ is of minimal genus. Since $K$ is Floer simple, there exists $\mathfrak s$ such that
\begin{eqnarray*}
\tau(Y, K,\mathfrak s)&=&A_{\text{topmost}}\\
&=&\frac{\langle c_1(\overline{\xi}),[F, \partial F]\rangle- |[\mu] \cdot [\partial F]|}{2|[\mu] \cdot [\partial F]|}\\
&=&\frac{-\chi(F) + |[\mu] \cdot [\partial F]|}{2|[\mu] \cdot [\partial F]|}\\
&=&\frac{y([F])}{|[\mu] \cdot [\partial F]|},
\end{eqnarray*}
where the last two equalities follow from Lemma~\ref{ChernClassFormula} and \cite[Theorem~1.1]{Ni2009}, respectively.
This completes the proof.
\end{proof}
We point out that to prove the second statement of Proposition~\ref{prop:TightnessGeneralization}, we do not use the fact that the rational Seifert surface is of minimal genus.

{\rmk \label{rem:Lens} Hedden in \cite[Proposition~2.1]{Hedden2010} shows that for a fibered knot $K$ in $S^3$, statements (1) and (2) of Theorem~\ref{thm:S1S2EquivState} are equivalent, and both are equivalent to having $\xi_K$ being tight. One key tool used in his proof is that there is only one unique tight contact structure $\xi_{\textup{std}}$ on $S^3$. Moreover, $\xi_{\textup{std}}$ is detected by the contact invariant, that is $c(\xi_{\textup{std}}) \ne 0$. This follows from the fact that the invariant associated to $(S^3, \xi_{\textup{std}})$ is equal to the generator of $\widehat{HF}(S^3) \cong \mathbb Z$. For L-spaces, there could be multiple tight contact structures, and some of them might not be detected by the contact invariant. However, for lens spaces it is known that all the tight contact structures are distinguished by the Heegaard Floer contact invariant. See~\cite[p.3]{Lisca2007} and \cite{Lisca1997}. In summary, for a fibered, Floer simple knot $K$ in a lens space $Y$, endowed with a minimal genus Seifert surface $F$, the following equivalent statements hold: (1) $\xi_K$ is tight, (2) $c(\xi_K)\not = 0$, and (3) there exists $\mathfrak s$ such that $\tau(Y, K,\mathfrak s)=\frac{y([F])}{|[\mu] \cdot [\partial F]|}$. This generalizes~\cite[Items (2), (3), and (4) of Proposition~2.1]{Hedden2010}.}

{\rmk \label{S1S2Overtwisted}In \cite[Lemma~1.18]{Baker2013}, it is proved that for a knot $L \subset S^1 \times S^2$, the exterior fibers over $S^1$ if and only if $L$ is isotopic to a spherical braid. Therefore, Theorem~\ref{thm:S1S2fiber} states that if a knot $L \subset S^1 \times S^2$ admits an L-space surgery, then its exterior fibers over $S^1$. Equivalently, there is a fibration of $S^1 \times S^2\setminus \nu^{\circ}(L)$ where the boundary of the fibers consists of the meridians of $L$. It can be proved that the contact structure compatible with the fibration is overtwisted unless the braid index of $L$, when $L$ is viewed as a spherical braid in $S^1 \times S^2$, is one. In the lack of an application of this result related to the purpose of the paper, we will not present a proof here, however, it will be interesting to investigate whether or not such contact structures can be classified (e.g. via the Hopf invariant~\cite{Eliashberg1989}).}


\section{Knots in $S^1 \times S^2$ with L-space surgeries}\label{sec:S1S2}

This section is devoted to the proof of Theorem~\ref{thm:S1S2fiber} and Proposition~\ref{prop:FloerSimple}. Let $L$ and $U$ be knots in $S^3$ such that $U$ is the unknot, and the linking number between $L$ and $U$ is $p>0$. Suppose that some surgery on the link $L\cup U$ results in an L-space $Y$, where the surgery slope on $U$ is zero. Let $K\subset Y$ be the dual knot of $L$ (i.e. $K$ is the core of the solid torus attached to $S^3_0(U)\setminus \nu^{\circ}(L)$). Let $\mu$ be the meridian of $K$. Let also $\left\{\mu_L, \lambda_L\right\}$ and $\left\{\mu_U, \lambda_U\right\}$ be the the meridian-longitude coordinates of $L$ and $U$ in $S^3$, respectively. Set $\alpha=\mu_L$.

If $M$ is a rational homology $S^1 \times D^2$ (e.g. the complement of a knot in a rational homology sphere), we say that $M$ is {\it semi-primitive} if the torsion subgroup of $H_1(M)$ is contained in the image of $\iota : H_1(\partial M) \to H_1(M)$, where $\iota$ is the map induced by inclusion.

\begin{proof}[Proof of Theorem~\ref{thm:S1S2fiber}]
We first show that $M=Y\setminus\nu^{\circ}(K)$ is semi-primitive.
It is well known that $H_1(S^3\setminus(L\cup U))$ is freely generated by $\mu_L,\mu_U$. Moreover, we have the equalities:
$$\lambda_L=p\mu_U,\quad \lambda_U=p\mu_L.$$
Since $M$ is obtained from $S^3\setminus \nu^{\circ}(L\cup U)$ by Dehn filling on $\partial\nu(U)$ with slope $\lambda_U$, we have
$$H_1(M)=\langle \mu_L,\mu_U|\:p\mu_L=0\rangle.$$
Hence the torsion subgroup of $H_1(M)$ is generated by $\mu_L$, which is contained in the image of $\iota$. This shows that $M$ is semi-primitive.

Since $L\subset S^3_0(U)$ admits an L-space surgery, \cite[Proposition 7.8]{Rasmussen2015} implies that $M$ is a generalized solid torus in the sense of~\cite[Definition~7.2]{Rasmussen2015}.
Now it follows from \cite[Corollary~7.12]{Rasmussen2015} that $Y$ fibers over $S^1$.
\end{proof}

We now turn to proving that $K$ is a Floer simple knot in $Y$. We recall from Section~\ref{intro} that $K\subset Y$ is called Floer simple if
\[
\rank \widehat{HFK}(Y, K) = \rank \widehat{HF}(Y).
\]
Let $\mathrm{Tor}_M \subset H_1(M)$ be the torsion subgroup. As in \cite{Rasmussen2015} we take the map
\[
\phi : H_1(M) \to \mathbb{Z},
\]
where $\phi$ is the projection from $H_1(M)$ to $\displaystyle H_1(M)/ \mathrm{Tor}_M \cong \mathbb{Z}$, and the isomorphism is chosen so that $\phi(\mu_U) > 0$. Combining \cite[Lemma~3.2 and Corollary~3.4]{Rasmussen2015} we get that:
\begin{prop}\label{prop:RasmussenFloerSimple} Let $K$ be a knot in an L-space $Y$ that admits a non-trivial L-space surgery. If $\phi(\mu) > ||M||$, where $||M||$ is the Thurston norm of a generator of $H_2(M, \partial M)$, then $K$ is Floer simple in $Y$.
\end{prop}

\begin{proof}[Proof of Proposition~\ref{prop:FloerSimple}]Using \cite[Proposition 7.8]{Rasmussen2015}, we conclude that any rational homology sphere obtained by surgery on $K$ is an L-space.
Hence in order to apply Proposition~\ref{prop:RasmussenFloerSimple} to the knot $K$, we only need to check that $\phi (\mu) > ||M||$. Let $F$ be a minimal genus rational Seifert surface for $K$ in $Y$. By Theorem~\ref{thm:S1S2fiber}, $F$ is a fiber of a fibration of $Y\setminus \nu^{\circ}(K)$ over $S^1$. Since the $\alpha$--surgery on $K$ yields $S^1\times S^2$, $F$ must be a punctured two-sphere. (We remind the reader that $\alpha = \mu_L$.) The number of punctures must be $p$, since $L$ and $U$ link each other $p$ times.
Therefore, $\chi(F) = 2-p$. Consequently, $||M|| = -\chi(F) = p-2$. It is just left to compute $\phi(\mu)$. Since $\mu\ne\mu_L$, we must have $\mu=n\mu_L+m\lambda_L$ for some $m\ne0$. As showed in the proof of Theorem~\ref{thm:S1S2fiber}, $\lambda_L = p\mu_U$ and $H_1(M)=\langle \mu_L,\mu_U|\:p\mu_L=0\rangle$, hence \[\phi(\mu) = \phi(n\mu_L+m\lambda_L)=\phi(m\lambda_L)=|m|p.\] Since $m\ne 0$, we get that $\phi(\mu) > ||M||$.
\end{proof}


\section{Knots in L-spaces with null surgery to fibered manifolds}\label{Thurston-norm}
The focus of this section is on proving Theorems~\ref{SurfaceBundle} and \ref{NormMinimizing}.
Let $K$ be a knot of order $p$ in a rational homology sphere $Y$, endowed with a framing $\lambda$.
As usual, let $F$ be a minimal genus rational Seifert surface for $K$. Note that $p$ is the intersection number of $\partial F$ with $\mu$. We define
\[g(K)=-\frac{\chi(F)}{2p}+\frac12\]
to be the {\it normalized genus} of $K$. When $K$ is null-homologous, that is when $p=1$, this descends to the standard definition of the three-genus of a null-homologous knot in $Y$.
The following fact about the connected sum of knots is elementary.

\begin{lemma}\label{lem:ConnSum}
Let $K_i$ be a knot in a rational homology sphere $Z_i$, $i=1,2$. Then
\begin{equation}\label{eq:GenusSum}
g(K_1\#K_2)=g(K_1)+g(K_2).
\end{equation}
Moreover, $K_1\#K_2$ is fibered if and only if both $K_1$ and $K_2$ are fibered.
\end{lemma}
\begin{proof}
Let $Z=Z_1\#Z_2$, and $S\subset Z$ be a sphere which splits $Z$ into a punctured $Z_1$ and a punctured $Z_2$. We may assume $S$ intersects  $K_1\#K_2$ exactly twice. Set $M_Z=Z\setminus\nu^\circ (K_1\#K_2)$.
The annulus $A=S\cap M_Z$ splits $M_Z$ into two submanifolds $M_{Z_1}$ and $M_{Z_2}$, such that $M_{Z_i}$ is homeomorphic to $Z_i\setminus\nu^\circ (K_i)$.

Let $p_i$ be the order of $[K_i]$ in $H_1(Y_i)$. The order of $[K_1\#K_2]$ will be $p=\mathrm{lcm}(p_1,p_2)$, the least common multiple of $p_1$ and $p_2$. Take $F_i\subset M_{Z_i}\cong Z_i\setminus\nu^\circ (K_i)$ to be a minimal genus rational Seifert surface for $K_i$. We can isotope $F_i$ so that $F_i\cap A$ consists of $p_i$ essential arcs in $A$. We may assume that $$(p_2F_1)\cap A=(p_1F_2)\cap A,$$ since each of $p_1 F_2$ and $p_2 F_1$ consists of $p_1p_2$ essential arcs in $A$. Here $p_2F_1$ denotes the union of $p_2$ parallel copies of $F_1$, and similarly for $p_1 F_2$. Thus $F=(p_2F_1)\cup(p_1 F_2)$ is a rational Seifert surface for $K_1\#K_2$. We get
\begin{eqnarray*}
g(K_1\#K_2)&\le&-\frac{\chi(F)}{2p_1p_2}+\frac12\\
&=&-\frac{p_2\chi(F_1)+p_1\chi(F_2)-p_1p_2}{2p_1p_2}+\frac12\\
&=&-\frac{\chi(F_1)}{2p_1}+\frac12-\frac{\chi(F_2)}{2p_2}+\frac12\\
&=&g(K_1)+g(K_2).
\end{eqnarray*}

On the other hand,
let $G$ be a minimal genus rational Seifert surface for $K_1\#K_2$. We may assume that $G$ is transverse to $A$. We may further assume $G\cap A$ consists of $p$ essential arcs in $A$, otherwise we can compress $G$ using the disk bounded by a circle in $A$ and replace $G$ with a new rational Seifert surface with genus smaller than or equal to the genus of $G$. Set $G_i=G\cap M_{Z_i}$. Each surface $G_i$ is a rational Seifert surface for $K_i$. It follows that
\begin{eqnarray*}
g(K_1)+g(K_2)&\le&-\frac{\chi(G_1)}{2p}+\frac12-\frac{\chi(G_2)}{2p}+\frac12\\
&=&-\frac{\chi(G)}{2p}+\frac12\\
&=&g(K_1\#K_2).
\end{eqnarray*}
This proves (\ref{eq:GenusSum}).

A careful look at the above argument will prove the second statement in the lemma. Suppose that both $K_1$ and $K_2$ are fibered. Let $\phi_i: M_{Z_i}\to S^1$ be a fibration with fiber surface $F_i$. The map $\phi_1^{p_{2}}: M_{Z_1}\to S^1$ is a fibration with fiber surface $p_{2}F_1$. Similarly, $\phi_2^{p_1}$ is a fibration of $M_{Z_2}$ with fiber surface $p_1 F_2$. We may assume $\phi_1^{p_2}|A=\phi_2^{p_1}|A$. Thus $\phi_1^{p_2}\cup\phi_2^{p_1}: M_{Z_1}\cup M_{Z_2}\to S^1$ defines a fibration of $M_Z$ over $S^1$. For the converse, suppose that $K_1\#K_2$ is fibered in $Z$. Let $\phi: M_Z \to S^1$ be a fibration with $G$ a fiber surface. Since $G\cap A$ consists of essential arcs, we may assume $\phi|A$ is a fibration. Therefore, $\phi|M_{Z_i}$ is a fibration for $i=1,2$.
\end{proof}

Recall that the null slope of $K$ is the unique isotopy class of the curve $\alpha$ in $\partial M$ that generates the kernel of the map $H_1(\partial M; \mathbb Q) \to H_1(M; \mathbb Q)$ induced by the inclusion map of $\partial M$ into $M$. Note that the class of $\alpha$, as an element of $H_1(\partial M)$, can be written as $\alpha= q'\mu + p'\lambda$, for some integers $q'$ and $p'>0$. Note also that $p$, the order of $[K]$ in $H_1(Y)$, is a multiple of $p'$.

A {\it Morse surgery} on $K$ is filling $M$ along a curve $m \cdot \mu + \lambda$, for some integer $m$. It is a well-known fact that Dehn surgery on $K$ with coefficient $q'/p'$ can be realized as Morse surgery with coefficient $m$ on the knot $K \# O_{p'/r}$ inside $Y \# L(p', r)$ where $q'=mp'-r$ with $0\le r <p'$, where $O_{p'/r}$ is the image of $K$ in $L(p',r)$ when $K$ is the unknot, $Y=S^3$, and the lens space $L(p',r)$ is obtained by performing $p'/r$ on the other component of the link in Figure~\ref{Morse}. This follows from the Slam-Dunk move. See~\cite[p. 501]{Cochran1988}. Let $\alpha'$ be the null slope on $K \# O_{p'/r}$, then $\alpha'$ is the framing with slope $m$.
We point out that in order to make sense of the surgery coefficient in our setting we first need to choose a longitude $\lambda$ for $K$. See Figure~\ref{Morse}.

\begin{figure}[t!]
 \begin{center}
 \subfigure[]
{
 \includegraphics[scale=.5]{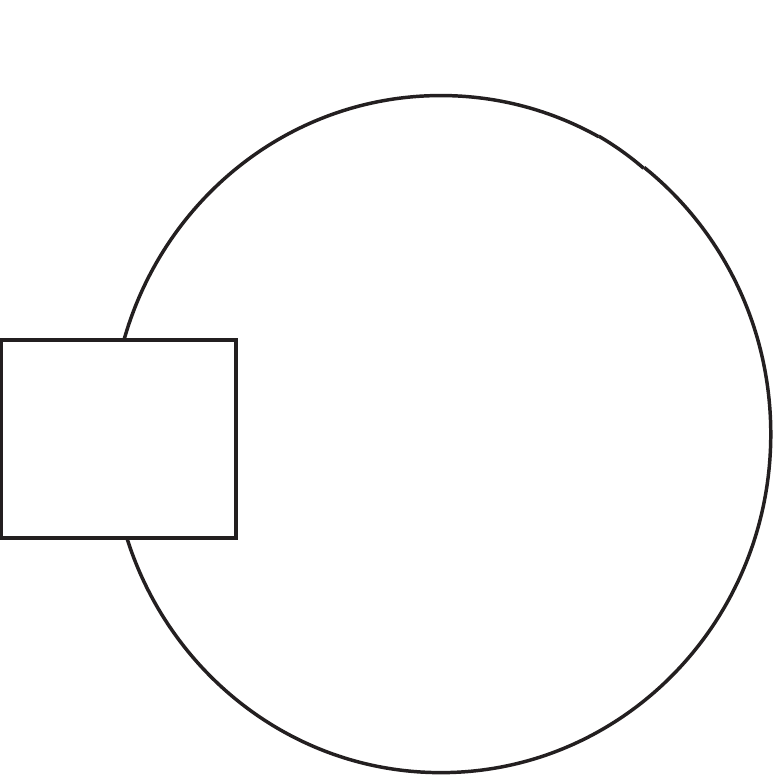}
\label{fig1:subfig1}
 }
\subfigure[]
{
  \includegraphics[scale=.5]{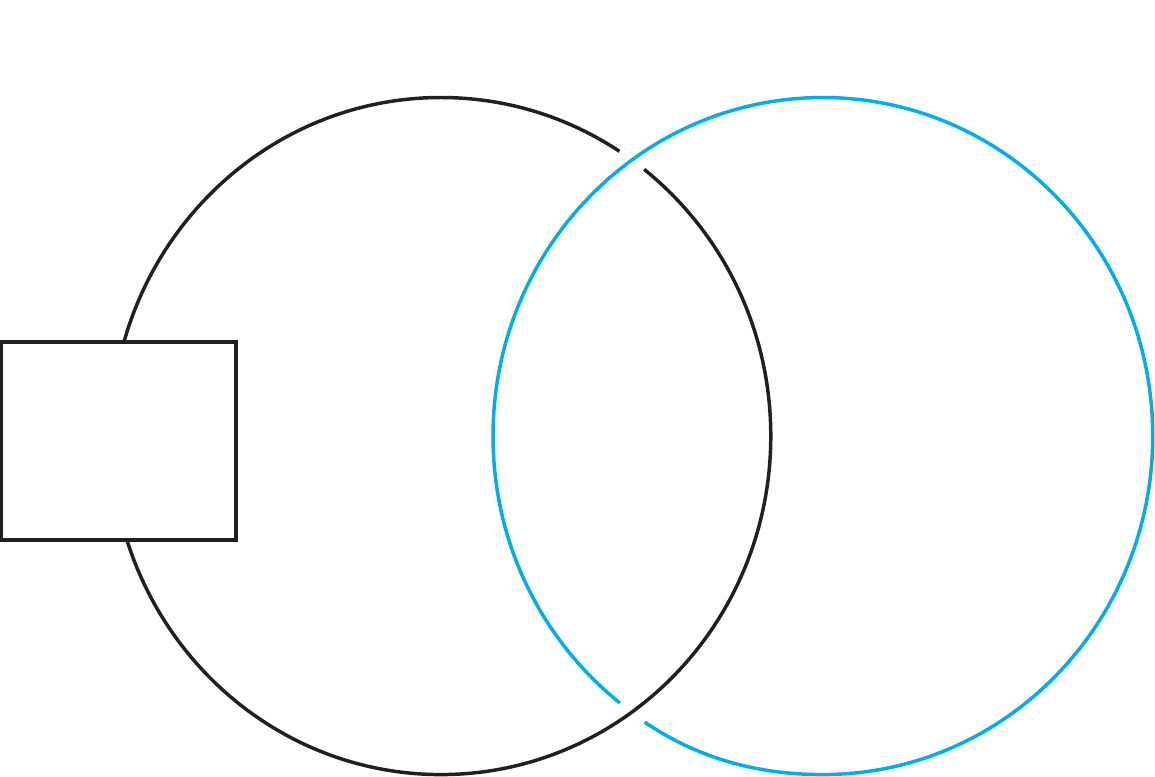}
   \label{fig1:subfig2}
}
\put(-285,41){\huge $K$}
\put(-285,91){$q'/p'$}
\put(-163,41){\huge $K$}
\put(-153,91){$m$}
\put(-12,91){$p'/r$}
\caption{\small{Dehn surgery on $K\subset Y$ can be realized as a Morse surgery on a knot inside the connected sum of $Y$ with a lens space. (a) Dehn surgery on $K\subset Y$ with coefficient $q'/p'$. (b) The corresponding Morse surgery on $K \# O_{p'/r}\subset Y\# L(p',r)$ with coefficient $m$ where $q'=mp'-r$ with $0\le r <p'$. The blue curve with its surgery coefficient $p'/r$ represents the lens space summand $L(p',r)$.}}
\label{Morse}
\end{center}
\end{figure}
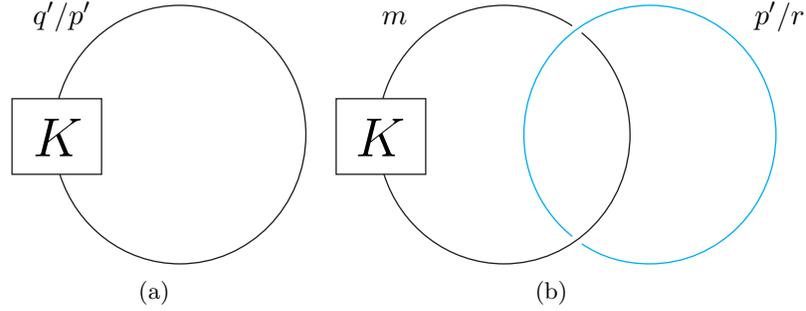

\begin{cor}\label{cor:ReduceToMorse}
Using the notation of this section,
\newline (a) If
$K \# O_{p'/r}\subset Y \# L(p',r)$ is fibered, then $K$ is fibered in $Y$.
\newline (b) Let $F$ be a minimal genus rational Seifert surface for $K$ and $\widehat F\subset Y_{\alpha}$ be the closed surface obtained by capping off $\partial F$ with disks. Then there exists a minimal genus rational Seifert surface $F'$ for $K'=K \# O_{p'/r} \subset Y \# L(p',r)$ such that $\chi(\widehat F)=\chi(\widehat{ F'})$, where $$\widehat{ F'}\subset Y_{\alpha}=(Y\# L(p',r))_{\alpha'}(K \# O_{p'/r})$$ is obtained from   $F'$ by capping off its boundary with disks.
\end{cor}
\begin{proof}
The statement (a) follows directly from Lemma~\ref{lem:ConnSum}. Thus, we only need to prove (b). Let $l=p/p'$ be the number of components of $\partial F$. Then \begin{equation}\label{chi1}\chi(\widehat F)=\chi(F)+l.\end{equation} Similar to the proof of Lemma~\ref{lem:ConnSum}, a Thurston norm minimizing rational Seifert surface for $K'$ may be obtained by gluing $F$ and $lD_{p'/r}$ along $p$ arcs. We call this rational Seifert surface surface $F'$. Then
 \begin{equation}\label{chi2}
\chi(F')=\chi(F)+l-p.
\end{equation}
 The order of $[K']$ in $H_1(Y\#L(p',r))$ is $p$, that is equal to the order of $[K] \in H_1(Y)$. Also, $[\mu]\cdot[\partial F']=p$. Combining these two facts we see that $F'$ is a minimal genus rational Seifert surface for $K'$.
By the discussion in the paragraph before this corollary, we have that the null slope of $K'$ is a framing. See Figure~\ref{Morse}. Hence, $\partial F'$ has exactly $p$ components. This observation, together with Equations~\eqref{chi1} and~\eqref{chi2} will give us that
\begin{align*}
\chi(\widehat{F'}) &=\chi(F')+p\\
                          &=\chi(F)+l\\
                          &=\chi(\widehat F). \qedhere
\end{align*}
\end{proof}

The main idea that will be used to prove Theorem~\ref{SurfaceBundle} is to compare the exact triangle of Theorem~\ref{ExactTriangle} with another exact sequence that differs only in one term with~\eqref{SurgeryTriangle}. The rest of the effort will be devoted to prove that those terms are also isomorphic. In Section~\ref{section:Background} we observed that for a relative Spin$^c$ structure $\xi$, $C_{\xi}= CFK^\infty(Y, K, \xi)$ is a chain complex. Moreover, every relative Spin$^c$ structure has an Alexander grading. We have the following short exact sequence
\begin{equation}\label{SES}
0\to C_{\underline{\xi}}\{i\ge 0\text{ and }j<1\} \to C_{\underline{\xi}}\{i\ge 0 \text{ or } j\ge 1\} \xrightarrow{h^+_{\underline{\xi},1}} C_{\underline{\xi}}\{j\ge 1\}\to 0,
\end{equation}
where $\underline{\xi} \in \underline{\mathrm{Spin}^c}(Y, K)$ is a relative Spin$^c$ structure with the least Alexander grading (see Equation~\eqref{AlexanderGrading}). We point out that $h^+_{\underline{\xi},1}$ is just the horizontal projection. Since $j\ge 1$ (instead of~$j\ge 0$), we use a different notation for the horizontal projection from that of~\eqref{projection}.

The goal of the next two lemmas is to replace the complexes in~\eqref{SES} with three other complexes so that, after taking homology, two out of three of the replaced terms will be the summands of the corresponding terms of~\eqref{SurgeryTriangle}.
{\lemma\label{CFKidentification} In the short exact sequence of~\eqref{SES}, we have
$$\displaystyle
H_*(C_{\underline{\xi}}\{i\ge 0\text{ and }j<1\}) \cong \widehat{HFK}(Y, K, \underline{\xi}),$$
where $\underline{\xi}$ is a Spin$^c$ structure with the least Alexander grading.}
\begin{proof}
We show that
\[
C_{\underline{\xi}}\{i\ge 0\text{ and }j<1\}=C_{\underline{\xi}}\{i= 0\text{ and }j=0\}.
\]
Taking homology of both sides gives us the statement of the lemma. Fix a doubly pointed Heegaard diagram $(\Sigma, \mbox{\boldmath${\alpha}$},\mbox{\boldmath$\beta$}, w, z)$ and a rational Seifert surface $F$ for the pair $(Y, K)$. Let $[\mathbf{x}, i, j]\in C_{\underline{\xi}}\{i\ge 0\text{ and }j<1\}$. Using Equation~\eqref{SpinStrFormula}, we get that
\[
\langle c_1(\underline{\xi}), [F,\partial F]\rangle = \langle c_1(\mathfrak{s}_{\mathbf{w},\mathbf{z}}(\mathbf x)+(i-j)PD[\mu]), [F,\partial F]\rangle .
\]
By Lemma~\ref{ChernClassFormula}, $\langle c_1(\underline{\xi}), [F,\partial F]\rangle=\chi(F)$. Therefore,
\begin{eqnarray*}
 \chi(F)  &=& \langle c_1(\mathfrak{s}_{\mathbf{w},\mathbf{z}}(\mathbf x) + (i-j) PD[\mu]),[F,\partial F] \rangle\\
&=& \langle c_1(\mathfrak{s}_{\mathbf{w},\mathbf{z}}(\mathbf x)), [F,\partial F]\rangle + 2(i-j) \langle PD[\mu],[F,\partial F] \rangle\\
&\ge& \chi(F) + 2(i-j)p,
\end{eqnarray*}
which in turn forces that $i=0$ and $j=0$. The last inequality, again, follows from Lemma~\ref{ChernClassFormula}.
\end{proof}
Recall the natural map $G_{Y, K}: \underline{\mathrm{Spin}}^c(Y, K) \to \mathrm{Spin}^c(Y)$ which associates an ``absolute" Spin$^c$ structure to every relative one in the knot exterior. Recall also $\Xi: \mathrm{Spin}^c(Y_n) \to \underline{\mathrm{Spin}^c}(Y, K)$, the map in Theorem~\ref{LargeSurgery}.
\begin{lemma}\label{lem:IdChain}
In the short exact sequence of Equation~\eqref{SES}, we have
\[
H_*(C_{\underline{\xi}}\{ i \ge 0 \text{ or } j \ge 1\}) \cong HF^+(Y_n, \mathfrak t) , \text{ and } H_*(C_{\underline{\xi}}\{j \ge 1\})\cong HF^+(Y,G_{Y,K}(\underline{\xi})+PD[\lambda]),
\]
where $\mathfrak t \in \mathrm{Spin}^c(Y_n)$ is a Spin$^c$ structure with $\Xi(\mathfrak t) = \underline{\xi} + PD[\mu]$, and $n\gg0$.
\end{lemma}
\begin{proof}
Using \cite[Proposition~3.2]{Ozsvath2010}, we get that
\begin{equation}\label{eq:VertShift}
C_{\underline{\xi}}\{ i \ge 0 \text{ or } j \ge 1\} \cong C_{\underline{\xi} + PD[\mu]}\{ i \ge 0 \text{ or } j \ge 0\},
\end{equation}
as $(\mathbb Z \oplus \mathbb Z)$--filtered chain complexes. This together with Theorem~\ref{LargeSurgery} will give us the first isomorphism. For the second isomorphism, we have
\[H_*(C_{\underline{\xi}}\{j \ge 1\})\cong H_*(C_{\underline{\xi}}\{j \ge 0\})\cong HF^+(Y,G_{Y,K}(\underline{\xi})+PD[\lambda]),\]
where the last isomorphism follows from the identification in~\eqref{eq:identification} together with Equation~\eqref{eq:Equivariant}.
\end{proof}

\noindent Similar to~\eqref{SpinEquivalence}, define
\[
[\xi]_{Y} = \left \{ \xi' \in \underline{\mathrm{Spin}}^c(Y, K) | G_{Y,K}(\xi') \in [G_{Y,K}({\xi})]_{Y}, A(\xi') = A({\xi})\right \},
\]
where $A$ denotes the Alexander grading defined in~\eqref{AlexanderGrading}.
Also, define
\[
 C_{[{\xi}]_{Y}}\{i\ge 0\text{ and }j<1\} = \bigoplus_{\xi' \in [{\xi}]_{Y}} C_{\xi'}\{i\ge 0\text{ and }j<1\}.
\]

\begin{prop}\label{prop:ZeroSurgIsom}
Let $K$ be a knot in an L-space $Y$, and $F$ be a minimal genus rational Seifert surface for $K$. Suppose that the null slope $\alpha$ of $K$ is a framing $\lambda$, and $g=g(F)>1$. Let $\underline{\xi}\in\mathrm{Spin}^c(Y,K)$ be a relative Spin$^c$ structure with the least Alexander grading. Let also $\mathfrak t$ be a Spin$^c$ structure on $Y_n$ with $\Xi(\mathfrak t) = \underline{\xi}+PD[\mu]$. Then we have the isomorphism
\[
\widehat{HFK}(Y, K, [\underline{\xi}]_{Y}) \cong HF^+(Y_\alpha, Q^{-1}([\mathfrak t]_{Y_{n}})).
\]
\end{prop}

We will first prove a technical lemma that will be useful to prove the proposition. The assumptions are the same as those of Proposition~\ref{prop:ZeroSurgIsom}.
Recall from Subsection~\ref{subsect:Triangle} that $[\mathfrak t]_{Y_n}\subset \mathrm{Spin}^c(Y_n)$ is the $\langle[\lambda]\rangle$--orbit that contains $\mathfrak t$, and $[\mathfrak s_{\mathfrak t}]_{Y}\subset \mathrm{Spin}^c(Y)$ is the unique $\langle[\lambda]\rangle$--orbit that is cobordant to $[\mathfrak t]_{Y_n}$ in $W'_n$. As in Subsection~\ref{subsect:Triangle}, $S$  denotes the core of the two-handle (attached to $Y$) in $W'_n$, which is a surface with boundary.

\begin{lemma}\label{lem:OneToOne}
Given $\mathfrak t'\in[\mathfrak t]_{Y_n}$, let $\varphi(\mathfrak t')=(\mathfrak x(\mathfrak t')+PD[S])|_Y$. Then $\varphi$ defines a one-to-one correspondence $[\mathfrak t]_{Y_n}\to [\mathfrak s_{\mathfrak t}]_{Y}$. Moreover, $\Xi$ defines a  one-to-one correspondence
$[\mathfrak t]_{Y_n}\to [\underline{\xi}]_Y+PD[\mu]$.
\end{lemma}
\begin{proof}
Viewing $\lambda$ as a curve in $M= Y\setminus \nu^{\circ}(K)$, its homology class represents an element in $H_1(M)$, $H_1(Y)$ and $H_1(Y_n)$. We will show that $[\lambda]$ has order $p$ in each of these three homology groups. Clearly $[\lambda]$ has order $p$ as an element of $H_1(Y)$, since $[\lambda]=[K]\in H_1(Y)$. Having assumed that the framing $\lambda$ is the same as the null slope of $K$, the order of $[\lambda]$, when viewed as an element of $H_1(M)$, is also $p$. Suppose that the order of $[\lambda]\in H_1(Y_n)$ is $0<r \le p$. That is, $r[\lambda]=[0]\in H_1(Y_n)$. Since $Y_n$ is obtained by gluing a solid torus to $M$ along $n\cdot \mu + \lambda$, we get that
\[r[\lambda]=s(n[\mu]+[\lambda])\in H_1(M),\]
for some integer $s$. Since $[\lambda]\in H_1(M)$ is a torsion element while $[\mu]$ is non-torsion, we must have $s=0$. Hence $r=p$.

By the definition of $\varphi$, $\varphi(\mathfrak t')$ is cobordant to $\mathfrak t'$ in $W'_n$, so $\varphi(\mathfrak t')\in [\mathfrak s_{\mathfrak t}]_{Y}$. As for $\Xi$, note that $\Xi(\mathfrak t)=\underline{\xi}+PD[\mu]$. Every $\mathfrak t'\in[\mathfrak t]_{Y_n}$ has the form $\mathfrak t+kPD[\lambda]$ for some integer $k$. Then, using Equation~\eqref{eq:Equivariant}, $G_{Y,K}(\Xi(\mathfrak t'))=G_{Y,K}(\underline{\xi}+PD[\mu])+kPD[\lambda]$.
Using Equation~\eqref{eq:AlexanderDifference},
\[
A(\underline{\xi}+PD[\mu])-A(\Xi(\mathfrak t')) = k\cdot \frac{\langle PD[\lambda], [F, \partial F]\rangle}{[\mu] \cdot [\partial F]} = 0.
\]
That is, $\Xi(\mathfrak t')\in  [\underline{\xi}]_Y+PD[\mu]$.

All three sets of Spin$^c$ structures in the lemma are affine spaces over $\langle[\lambda]\rangle\cong\mathbb Z/p\mathbb Z$.
Moreover, both maps $\varphi$ and $\Xi$ are equivariant with respect to the action of $PD[\lambda]$. Our conclusion then follows.
\end{proof}

\begin{proof}[Proof of Proposition~\ref{prop:ZeroSurgIsom}]
We recall that $F$ in the long exact sequence \eqref{SurgeryTriangle} is induced by Equation~\eqref{eq:F3}, which can be rewritten as
\[
\sum_{\mathfrak t'\in[\mathfrak t]_{Y_n}}f_{\mathfrak t'},
\]
where
\begin{equation}\label{eq:F3'}
f_{\mathfrak t'}=\sum_{k\in\mathbb Z}f^+_{W'_n(K),\mathfrak x(\mathfrak t')+kPD[S]}.
\end{equation}
Here $\mathfrak x=\mathfrak x(\mathfrak t') \in \mathrm{Spin}^c(W'_n(K))$ is as in Theorem~\ref{LargeSurgery}.

Fix $\mathfrak t'\in[\mathfrak t]_{Y_n}$, and let $\xi'=\Xi(\mathfrak t')$. Under the identifications in Lemma~\ref{lem:IdChain},
the maps $v^+_{\xi'}$ and $h^+_{\xi'}$ correspond to Spin$^c$ structures $\mathfrak x$ and $\mathfrak x+PD[S]$ on the two-handle cobordism $W'_n$, respectively. Note that the class $[S]$ represents an element in $H_2(W'_n, \partial W'_n)$.


Using the degree shift formula~\eqref{degreeshift} we see that the difference of the degrees of $f^+_{W'_n(K),\mathfrak x(\mathfrak t')}$ and $f^+_{W'_n(K),\mathfrak x(\mathfrak t')+kPD[S]}$ is
\begin{equation}\label{eq:SquareDiffernce}
\frac14(c_1(\mathfrak x)^2 - c_1(\mathfrak x+kPD[S])^2) = -k^2PD[S]^2 - k \langle c_1(\mathfrak x)\cup PD[S], [W'_n,\partial W'_n]\rangle ,
\end{equation}
where $[W'_n, \partial W'_n]$ is the fundamental class of $W'_n$. Since $H^2(W'_n(K);\mathbb Q)\cong\mathbb Q$, there exists a rational number $r$ with the property that
\begin{equation}\label{eq:Multiple}
c_1(\mathfrak x)= r\cdot PD[S].
\end{equation}
Let $\widetilde F$ be the capped off rational Seifert surface in $W'_n$. Then $H_2(W'_n(K);\mathbb Q)$ is generated by $[\widetilde F]$. Therefore,
\begin{equation}\label{eq:Multiple2}
\displaystyle r=\frac{\langle c_1(\mathfrak x), [\widetilde F]\rangle}{\langle PD[S], [\widetilde F]\rangle}.
\end{equation}

\noindent By Lemma~\ref{ChernClassFormula},
we have
\[\langle c_1(\underline{\xi}+PD[\mu]),[F,\partial F]\rangle=\chi(F)+2p=\chi(\widetilde F)+p.
\]
Lemma~\ref{lem:c1Evaluation} then implies that
\[\langle c_1(\mathfrak x), [\widetilde F]\rangle=\chi(\widetilde F)-pn.
\]
This, combined with Equations~(\ref{eq:PD[S]}, \ref{eq:Multiple}, \ref{eq:Multiple2}), we get that the difference in (\ref{eq:SquareDiffernce}) is
\begin{eqnarray}
\frac14(c_1(\mathfrak x)^2 - c_1(\mathfrak x+kPD[S])^2) &=& -k^2 PD[S]^2(1+ \frac{\langle c_1(\mathfrak x), [\widetilde F]\rangle}{k\langle PD[S], [\widetilde F]\rangle})\nonumber\\
                                                                                  &=&  -k^2 PD[S]^2(1+ \frac{\chi(\widetilde F) -pn}{knp})\nonumber\\
                                                                                  &=& -k^2 PD[S]^2(1-\frac1k+\frac{2-2g}{knp}).\label{eq:EvalDiff}
\end{eqnarray}
Assume that $n\gg0$. Since $W'_n(K)$ is a negative definite four-manifold, we see that $-k^2 PD[S]^2>0$. Also, $1-\frac1k>0$ unless $k=1$. So the right hand side of (\ref{eq:EvalDiff}) is positive provided that $k\ne1$. It is negative when $k=1$ and $g>1$. That is, when $g>1$, $v^+_{\xi'}$ has degree lower than that of $h^+_{\xi'}$, but higher than any of other terms in \eqref{eq:F3'}. In other words, the map in~\eqref{eq:F3'} has the form
\begin{equation}\label{eq:LowerOrder}
h^+_{\xi'}+\text{lower order terms}.
\end{equation}

\noindent Since $Y$ is an L-space, $h^+_{\xi'}$ induces a surjective map in homology. Lemma~\ref{lem:OneToOne} then implies that
\[
(h^+_{[\underline{\xi}+PD[\mu]]_{Y}})_*=(\bigoplus_{\xi' \in [\underline{\xi}+PD[\mu]]_{Y}} h^+_{\xi'})_*: HF^+(Y_n, [\mathfrak t]_{Y_n})\to HF^+(Y, [\mathfrak s_{\mathfrak t}]_{Y})
\]
is surjective.

\noindent Using Corollary~\ref{cor:DiffTors}, we see that the degrees of $h^+_{\xi_1}$ and $h^+_{\xi_2}$ are equal for any two $\xi_1,\xi_2\in [\underline{\xi}+PD[\mu]]_{Y}$. It follows from (\ref{eq:LowerOrder}) that
\[
F=(h^+_{[\underline{\xi}+PD[\mu]]_{Y}})_*+\text{lower order terms}.
\]
Since $(h^+_{[\underline{\xi}+PD[\mu]]_{Y}})_*$ is also surjective,
a standard algebraic argument implies that $F$ is surjective. Moreover,
\begin{equation}\label{eq:KerIsom}
\ker F\cong \ker(h^+_{[\underline{\xi}+PD[\mu]]_{Y}})_*.
\end{equation}

\noindent Using the exact sequence of~\eqref{SES}, we get the short exact sequence
\begin{equation}\label{Summand}
0\to C_{[\underline{\xi}]_{Y}}\{i\ge 0\text{ and }j<1\} \to C_{[\underline{\xi}]_{Y}}\{i\ge 0 \text{ or } j\ge 1\} \xrightarrow{\bigoplus h^+_{\xi,1}} C_{[\underline{\xi}]_{Y}}\{j\ge 1\}\to 0,
\end{equation}
where the direct sum on the second map is taken over all $\xi \in [\underline{\xi}]_{Y}$. It follows from~\eqref{eq:VertShift} that $h^+_{\xi,1}=h^+_{\xi+PD[\mu]}$. So the second map in~\eqref{Summand} is $h^+_{[\underline{\xi}+PD[\mu]]_{Y}}$.

We compare the exact sequence induced from (\ref{Summand}) with the exact sequence of Theorem~\ref{ExactTriangle}. In the latter sequence, we only need to take the orbit of the Spin$^c$ structure on $Y_n$ that gets mapped to $\underline{\xi}+PD[\mu]$ under $\Xi$.
Since both $(h^+_{[\underline{\xi}+PD[\mu]]_{Y}})_*$ and $F$ are surjective, our conclusion follows from Lemma~\ref{CFKidentification}, Lemma~\ref{lem:IdChain}, and (\ref{eq:KerIsom}).
\end{proof}

\begin{proof}[Proof of Theorem~\ref{SurfaceBundle}]
We first deal with the case that $\alpha$ is a framing. Let $g$ be the genus of a minimal genus rational Seifert surface for $K$.
If $g>1$, the assumption that $Y_\alpha$ fibers over the circle  together with \cite[Theorem~5.2]{Ozsvath2004e} will give that
\[\bigoplus_{\mathfrak s\in\mathrm{Spin}^c(Y_{\alpha}),\langle c_1(\mathfrak s),[\widehat{F}]\rangle=\chi(\widehat{F})}HF^+(Y_\alpha, \mathfrak t) \cong \mathbb Z.
\]
Therefore, Proposition~\ref{prop:ZeroSurgIsom} implies that
\[
\bigoplus_{\xi\in\mathrm{Spin}^c(Y_{\alpha}),\langle c_1(\xi),[F,\partial F]\rangle=\chi({F})}\widehat{HFK}(Y, K, \xi) \cong \mathbb Z.
\]
 Using \cite[Theorem~2.3]{Ni2014}, $K$ is fibered. For the case $g =1$, we need to use the twisted version of the exact triangle of \eqref{SurgeryTriangle}. All the steps are analogous to the proof for the case $g>1$. See~\cite{Ni2009b} where the exact triangle is obtained for a null-homologous knot. Finally, using Theorem~\ref{thm:S1S2fiber} for the case $g = 0$, the result follows.

If $\alpha$ is not a framing, by the paragraph before Corollary~\ref{cor:ReduceToMorse}, $Y_\alpha$ can be obtained by performing a Morse surgery on $K\#O_{p'/r}$ in the L-space $Y\#L(p',r)$. The previous case implies that $K\#O_{p'/r}$ is fibered. Hence, using Corollary~\ref{cor:ReduceToMorse}, $K$ is fibered.
\end{proof}

\begin{proof}[Proof of Theorem~\ref{NormMinimizing}]
Similar to the proof of Theorem~\ref{SurfaceBundle}, we first deal with the case that $\alpha$ is a framing. Let $F$ be a Thurston norm minimizing rational Seifert surface for $K$. Without loss of generality, we may assume $F$ is of minimal genus. If $g(F)\le1$, $\widehat{F}$ is a sphere or torus, hence must be Thurston norm minimizing. If $g(F)>1$, Lemma~\ref{ChernClassFormula} implies that there exists $\underline \xi\in\underline{\mathrm{Spin}^c}(Y,K)$ such that
\[
\widehat{HFK}(Y,K,\underline \xi)\ne0 \text{ and } \langle c_1(\underline \xi),[F,\partial F]\rangle=\chi(F).
\]
Proposition~\ref{prop:ZeroSurgIsom} implies that  $HF^+(Y_\alpha, \mathfrak s)\ne0$ for some $\mathfrak s\in\mathrm{Spin}^c(Y_\alpha)$ with $\langle c_1(\mathfrak s),[\widehat F]\rangle=\chi(\widehat F)$. Hence $\widehat F$ is Thurston norm minimizing by the adjunction inequality \cite[Theorem~7.1]{Ozsvath2004a}.

If $\alpha$ is not a framing, as before, $Y_\alpha$ can be obtained by performing a Morse surgery on $K\#O_{p'/r}$ in $Y\#L(p',r)$.
Let $F'$ be the minimal genus rational Seifert surface for $K\#O_{p'/r}$ as constructed in Corollary~\ref{cor:ReduceToMorse}. Let also $\widehat {F'}$ be its extension to the $m$--surgery on $K\#O_{p'/r}$ in $Y\#L(p',r)$. From the previous case, we know that $\widehat {F'}$ is Thurston norm minimizing. Hence, using Part~(2) of Corollary~\ref{cor:ReduceToMorse}, $\widehat {F}$ is also Thurston norm minimizing.
\end{proof}


\section{Directions for future research}
\subsection{Floer simple knots in L-spaces and fiberedness}
Let $K\subset Y$ be a knot in an L-space $Y$ that admits some $S^1 \times S^2$ surgery. We showed in Theorem~\ref{thm:S1S2fiber} that the complement of $K$ in $Y$ fibers over the circle. Using \cite[Proposition~7.8]{Rasmussen2015}, we conclude that every Morse surgery on $K$ (except for the one that results in $S^1 \times S^2$) will result in an L-space. As pointed out in the introduction if $Y=S^3$, then any knot with an L-space surgery will be fibered. For an arbitrary L-space $Y$, however, this is not always the case. Lidman and Watson in \cite{Lidman2012} constructed examples of non-fibered knots in L-spaces with L-space surgeries. It is known that if a Floer simple knot $K$ in an L-space is primitive, and the knot complement is irreducible, then $K$ is fibered~\cite[Theorem~6.5]{Boyer2012}\footnote{\cite[Theorem~6.5]{Boyer2012} is stated for a primitive knot in a lens space. The same proof can be applied to get the theorem we stated.}. Recall that a knot $K\subset Y$ is primitive if $[K] \in H_1(Y)$ is a generator. Let $[K]^{\perp}$ denote the orthogonal complement of the homology class $[K]\in H_1(Y)$ with respect to the linking form of $Y$. With having the notation of this section in place, we can reformulate that theorem as follows:
\begin{thm} \label{GeneralFiberedness}
Let $Y$ be an L-space, $K\subset Y$ be a Floer simple knot with irreducible complement. If $[K]^{\perp} = 0$, then $K$ is fibered.
\end{thm}

Note that we are replacing the primitiveness assumption by a criterion regarding the linking form of $Y$. We briefly review the classical notion of linking forms here. For a more detailed discussion, see~\cite{Melvin2010}, for instance.

\begin{defn}
\label{LinkingForm} The \emph{linking form} of a closed three-manifold $Y$ is the non-degenerate form
\[
lk_{Y} : \mathrm{Tor}_{Y} \times \mathrm{Tor}_{Y} \to \mathbb{Q}/\mathbb{Z}
\]
on the torsion subgroup $\mathrm{Tor}_{Y}$ of $H_1(Y)$ defined by $lk_{Y}(a, b) = \alpha \cdot \tau/n$, where $\alpha$ is any 1-cycle representing $a$ and $\tau$ is any 2-chain bounded by a positive integer multiple $n\beta$ of a 1-cycle $\beta$ representing $b$.
\end{defn}

If $Y$ is surgery on a framed link $L$, then $lk_{Y}$ is computed from the linking matrix $A$ of $L$, with framings on the diagonal, as follows. First use a change of basis to transform $A$ into a block sum $\mathbb{O} \oplus \mathbb{A}$. Here, $\mathbb{O}$ is a zero matrix and $\mathbb{A}$ is nonsingular. This corresponds to a sequence of handle slides in the Kirby diagram \cite{Gompf1999}, transforming $L$ into $L_{\mathbb{O}} \cup L_{\mathbb{A}}$. Now following~\cite{Seifert}, the linking form $lk_{Y}$ is presented by the matrix $\mathbb{A}^{-1}$ with respect to the generators of $\mathrm{Tor}_{Y}$ given by the class of the meridians of the components of $L_{\mathbb{A}}$.

To see that Theorem~\ref{GeneralFiberedness} is a reformulation of \cite[Theorem~6.5]{Boyer2012}, we start by the following lemma about $[K]^{\perp}$:

\begin{lemma}\label{lem:DisjointSurf}
Suppose that $K$ is a knot in a rational homology sphere $Y$ and $\overline a\in[K]^{\perp}$. Then there exists a knot $L$ in the complement of $K$, such that $[L]=\overline a$ and $L$ bounds a rational Seifert surface which is disjoint from $K$.
\end{lemma}
\begin{proof}
Let $p$ be the order of $[K]$ in $H_1(Y)$. There exists a rational Seifert surface $F$ for $K$ so that the intersection number of $\partial F$ with the meridian of $K$ is $p$. Let $L'\subset Y$ be a knot representing $\overline a$. We may assume $L'$ is disjoint from $K$. Since $lk_{Y}([K],\overline a)=0$, the algebraic intersection number of $L'$ with $F$ is a multiple of $p$. Performing connected sums of $L'$ with copies of the meridian of $K$, we can get a new knot $L$ disjoint from $K$, so that $L$ still represents $\overline a$ and the algebraic intersection number of $L$ with $F$ is zero. Hence any rational Seifert surface $G$ for $L$ has algebraic intersection number zero with $K$. Consequently, by removing the intersection points of $G$ with $K$ by adding tubes to $G$, we get a rational Seifert surface for $L$ that is disjoint from $K$.
\end{proof}

\noindent Lemma~\ref{lem:DisjointSurf} yields the following elementary characterization of primitive knots.

\begin{prop}\label{prop:Primitive}
Suppose that $K$ is a knot in a rational homology sphere $Y$, $M=Y\setminus\nu^{\circ}(K)$. Then the following three conditions are equivalent:
\newline(i) $K$ is primitive.
\newline(ii) $H_1(M)\cong\mathbb Z$.
\newline(iii) $[K]^{\perp}=0$.
\end{prop}
\begin{proof}
\noindent(i)$\Leftrightarrow$(ii). By definition, $K$ is primitive is equivalent to the condition that the map $\iota_K: H_1(K)\to H_1(Y)$ is surjective. Using the Mayer--Vietoris sequence for the pair $(Y,K)$, we see that the surjectivity of $\iota_K$ is equivalent to $H_1(Y,K)=0$. We have $H_1(M)\cong\mathbb Z\oplus\mathrm{Tor}_M$. By the Universal Coefficients Theorem, $\mathrm{Tor}_M$ is isomorphic to $H^2(M)$, which is (by Poincar\'e duality) isomorphic to $H_1(Y,K)$. Hence $K$ is primitive is equivalent to $H_1(M)\cong\mathbb Z$.

\noindent(ii)$\Rightarrow$(iii). Suppose that $\overline a\in[K]^{\perp}$ and let $L$ be a knot as in Lemma~\ref{lem:DisjointSurf}. Then $L$ represents a torsion element in $H_1(M)$. Since $H_1(M)\cong\mathbb Z$, $L$ is null-homologous in $M$. Hence, it is also null-homologous in $Y$. This means $\overline a=0$.

\noindent(iii)$\Rightarrow$(ii). If $H_1(M)\not\cong\mathbb Z$, then $H_1(M)$ contains a nonzero torsion element $a$. Let $L\subset M$ be a knot representing $a$. Then $L$ has a rational Seifert surface in $M$. Let $\overline a\in H_1(Y)$ be represented by $L$. By definition, $\overline a\in [K]^{\perp}$. Consider the long exact sequence $$0\to H_2(Y,M)\to H_1(M)\to H_1(Y).$$
Since $a$ is torsion, it is not contained in the image of $H_2(Y,M)\cong\mathbb Z$. So $\overline a$, being the image of $a$ in $H_1(Y)$, is nonzero. This contradicts the assumption that $[K]^{\perp}=0$.
\end{proof}

Recall from Section~\ref{sec:S1S2} that a rational homology solid torus $M$ is semi-primitive if $\mathrm{Tor}_M$ is contained in the image of $\iota : H_1(\partial M) \to H_1(M)$.
Similar to Proposition~\ref{prop:Primitive}, we have a characterization of semi-primitiveness in terms of the linking form of $Y$.

\begin{prop}
\label{Semiprimitive} Let $K$ be a knot in a rational homology sphere $Y$, $M=Y\setminus\nu^{\circ}(K)$. Then $M$ is semi-primitive if and only if $[K]^{\perp} \subset \langle [K] \rangle$.
\end{prop}
\begin{proof}
First assume that $[K]^{\perp} \subset \langle [K] \rangle$. For $a \in \mathrm{Tor}_M$, let $L\subset M$ be a knot representing $a$. Let also $F$ be a rational Seifert surface that $L$ bounds in $M$. Note that $F\subset Y$ is disjoint from $K$. Let $\overline a\in H_1(Y)$ be represented by $L$. By the definition of the linking form, $\overline a \in [K]^{\perp}$. Using the assumption $[K]^{\perp} \subset \langle [K] \rangle$, we get that $\overline a$ is homologous to $r[K]$ for some integer $r$. Therefore, there exists an oriented surface $F'$ that is co-bounded, on one side by $L$, and on the other side by $r$ parallel copies of $K$. Note that $F' \cap M$ is a surface $F''$ with boundary consisting of $\partial F'$ and meridian circles of $K$ which come from the intersection of $K$ with $F'$. Thus $[L]$ is homologous to $r[K]$ plus a multiple of the meridian of $K$ through the surface $F''$.
Consequently, $a$ can be written as a sum of $r$ longitudes of $K$ with a multiple of the meridian of $K$ in $H_1(M)$. That is, $a \in \text{ im}(\iota : H_1(\partial M) \to H_1(M))$.

Now suppose that $M$ is semi-primitive. Take an element $\overline a\in[K]^{\perp}$. Let $L$ be a knot as in Lemma~\ref{lem:DisjointSurf}. Suppose that $L$ represents $a\in H_1(M)$. Since $L$ has a rational Seifert surface in $M$, we get $a \in \mathrm{Tor}_M$. By the semi-primitiveness of $M$, $a$ is in the image of $\iota$. So $L$ cobounds an oriented surface $F'$ with a curve $L'$ in $\partial M$. Consider the map $H_1(M)\to H_1(Y)$ and observe that the image of $H_1(\partial M)$ in $H_1(Y)$ is generated by $[K]$. Thus, we get that $\overline a=[L']$ is a multiple of $[K]$ in $H_1(Y)$.
\end{proof}

Recall that if a knot $L\subset S^1 \times S^2$ admits a Dehn surgery to an L-space $Y$, then $L$ is fibered (Theorem~\ref{thm:S1S2fiber}). As in the proof of Theorem~\ref{thm:S1S2fiber}, $M=(S^1\times S^2)\setminus\nu^{\circ}(L)$ is semi-primitive. By Proposition~\ref{Semiprimitive},
the dual knot $K\subset Y$ has the property that $[K]^{\perp}\subset\langle [K]\rangle$. In the light of this fact, together with Theorem~\ref{GeneralFiberedness}, we make the following conjecture:
\begin{conj} \label{conj}
Let $Y$ be an L-space, and let $K\subset Y$ be a Floer simple knot with irreducible complement. If $[K]^{\perp} \subset \langle [K] \rangle$, then $K$ is fibered.
\end{conj}

Using Proposition~\ref{Semiprimitive}, we see that Conjecture~\ref{conj} could be equivalently stated as a generalization of \cite[Theorem~6.5]{Boyer2012} where the primitiveness assumption is replaced by the semi-primitiveness of the knot.

\subsection{``Positivity" of knots in $S^1\times S^2$ admitting L-space surgeries}
In another direction, it is known that for a knot $K \subset S^3$ with some L-space surgery, $K$ is a \emph{strongly quasipositive knot}. Let $B_n$ denote the braid group on $n$ strands, with generators $\sigma_1, \sigma_2, \cdot \cdot \cdot \sigma_{n-1}$. A strongly quasipositive link is a link that can be realized as the closure of the braid word
\[
\beta = \prod_{k=1}^m \sigma_{i_k, j_k},
\]
where $\sigma_{i,j}$ is of the form \begin{equation}\label{SQP}(\sigma_i \cdot \cdot \cdot \sigma_{j-2})\sigma_{j-1}(\sigma_i \cdot \cdot \cdot \sigma_{j-2})^{-1}.\end{equation}
There is a weaker notion of positivity called {\it quasipositivity} where the braid word $\beta$ is the multiple of arbitrary conjugates of positive generators in $B_n$ (whereas strongly quasipositive knots require these conjugates to be of a special form). That is, for quasipositive links, $(\sigma_i \cdot \cdot \cdot \sigma_{j-2})$ in~\eqref{SQP} is replaced by an arbitrary braid word. There is a more geometric, yet equivalent, definition of quasipositive links. Every such a link is a transverse $\mathbb C$-link, that is, it arises as the transverse intersection of $S^3 \subset \mathbb C^2$ with a complex plane curve $f^{-1}(0)\subset \mathbb C^2$, where $f$ is a non-constant polynomial. Algebraic links of singularities form a proper subfamily of quasipositive links. See, for instance, \cite{Boileau2001, Hedden2010, Rudolph1983}. For a non-null-homologous knot $L\subset S^1 \times S^2$ with fibered exterior we know that $L$ is isotopic to a spherical braid \cite[Lemma~1.18]{Baker2013}.
{\question \label{question}Given a knot $L\subset S^1 \times S^2$ that admits an L-space surgery, is there a notion of positivity for $L$ as a spherical braid?}
\\

The fact that the knot $L$ is indeed a spherical braid follows from Theorem~\ref{thm:S1S2fiber}. It should be noted that considering the obvious notion of positivity, by assigning a sign to each crossing, does not work in this setting. For if $\sigma$ is a generator of the spherical braid group, then $\sigma \sim \sigma^{-1}$. Here ``$\sim$" denotes an isotopy between braids.

\bibliographystyle{amsalpha2}

\bibliography{Reference}

\end{document}